\documentclass[12pt]{article}

\usepackage{amsmath}
\usepackage{amssymb}
\usepackage{mathtools}

\usepackage[blocks]{authblk}

\usepackage{hyperref}
\usepackage{siunitx}
\usepackage{booktabs}

\usepackage[a4paper, left=3.2cm,top=2.5cm,right=3.5cm,bottom=3.7cm]{geometry}
\usepackage[sort&compress,numbers]{natbib}

\usepackage{concmath}
\usepackage[T1]{fontenc}

\usepackage{enumitem}
\setitemize{label=\scriptsize{$\blacksquare$},topsep=0em,itemsep=0em}
\setenumerate{label=(\alph*),topsep=0em,itemsep=0em}

\usepackage{algorithmic}
\usepackage[ruled]{algorithm}

\DeclareMathOperator*{\argmin}{arg\,min}
 \DeclareMathOperator{\ran}{ran}

\usepackage{hyperref}
\usepackage{graphicx}
\usepackage{amssymb,amsmath, amscd,amsfonts,amsthm,mathtools}
\usepackage[utf8]{inputenc}

\usepackage[mathscr]{eucal}
\usepackage{todonotes}
\usepackage{pifont}

\usepackage{pgfplots}
\pgfplotsset{compat=1.8}

\makeatletter
\newcommand*\bigcdot{\mathpalette\bigcdot@{.6}}
\newcommand*\bigcdot@[2]{\mathbin{\vcenter{\hbox{\scalebox{#2}{$\m@th#1\bullet$}}}}}
\makeatother

\theoremstyle{theorem}
\newtheorem{theorem}{Theorem}[section]
\newtheorem{lemma}[theorem]{Lemma}
\newtheorem{proposition}[theorem]{Proposition}

\theoremstyle{definition}
\newtheorem{definition}[theorem]{Definition}

\newcommand{\R}{\mathbb{R}}
\newcommand{\Z}{\mathbb{Z}}
\newcommand{\N}{\mathbb{N}}
\newcommand{\dd}{\mathrm{d}}

\newcommand{\sph}{\mathbb S}
\newcommand{\set}[1]{\{#1\}}
\newcommand{\abs}[1]{\lvert#1\rvert}
\newcommand{\norm}[1]{\lVert#1\rVert}
\newcommand{\n}{\omega}
\newcommand{\inner}[2]{{#1}\bigcdot {#2}}
\newcommand\ip[2]{\left\langle#1,#2\right\rangle}
\newcommand{\kl}[1]{\left(#1\right)}

\newcommand{\noisenum}{{\boldsymbol \xi}}
\newcommand{\fnum}{{\boldsymbol f}}
\newcommand{\pnum}{{\boldsymbol p}}
\newcommand{\qnum}{{\boldsymbol q}}
\newcommand{\gnum}{{\boldsymbol g}}
\newcommand{\unum}{{\boldsymbol u}}

\newcommand{\Lnum}{\boldsymbol{L}}

\newcommand{\Cnum}{\boldsymbol{C}}

\newcommand{\Dnum}{\boldsymbol{D}}
\newcommand{\Inum}{\boldsymbol{I}}
\newcommand{\Pnum}{\boldsymbol{P}}
\newcommand{\Tnum}{\boldsymbol{\Phi}}

\newcommand{\RR}{R_0}
\newcommand{\pp}{z}
\newcommand{\Do}{\mathbf{D}}
\newcommand{\RC}{\mathcal{C}}
\newcommand{\cone}{S}
\newcommand{\cset}{M}

\newcommand{\f}{f}
\newcommand{\g}{g}
\newcommand{\U}[1]{U\left(#1\right)}
\newcommand{\Un}[1]{U_n(#1)}

\newcommand{\ds}{\dd S}
\newcommand{\Qm}{ Q}
\newcommand{\edot}{\,\cdot\,}
\newcommand{\lk}{{\ell,k}}
\newcommand{\Cn}{C_\ell^{(n-2)/2}}
\newcommand{\K}{K}
\newcommand{\F}{F}
\newcommand{\la}{\lambda}
\newcommand{\La}{\Lambda}
\newcommand{\al}{\alpha}
\newcommand{\ph}{\varphi}

\newcommand{\bkl}[1]{\left(#1\right)}

\newcommand{\xx}{D_{\mathbb X}}
\newcommand{\yy}{D_{\mathbb Y}}
\newcommand{\XX}{\mathbb X}
\newcommand{\YY}{\mathbb Y}

\newcommand{\tik}{\Phi}
\newcommand{\reg}{\Psi}

\newcommand{\dom}[1]{\Delta(#1)}

\usepackage{soul}
\usepackage{xcolor}
\colorlet{lred}{red!40}
\colorlet{lblue}{blue!40}
\colorlet{lgreen}{green!40}

\numberwithin{figure}{section}
\numberwithin{equation}{section}

\allowdisplaybreaks

\title{Variational  regularization  of the weighted conical Radon transform}

\date{August 4, 2018}

\author{Markus Haltmeier}
\affil{University of Innsbruck, Department of Mathematics\authorcr
Technikerstrasse 13, A-6020 Innsbruck, Austria\authorcr
E-mail: \texttt{markus.haltmeier@uibk.ac.at}}

\author{Daniela Schiefeneder}
\affil{University of Innsbruck, Department of Mathematics\authorcr
Technikerstrasse 13, A-6020 Innsbruck, Austria\authorcr
E-mail: \texttt{daniela.schiefeneder@uibk.ac.at}}
\begin{document}
\maketitle

\begin{abstract}
Recovering a function from integrals over conical surfaces recently got significant interest. It is relevant for emission tomography with Compton cameras and other imaging applications. In this paper, we consider the weighted conical Radon transform with vertices on the sphere.  Opposed to previous works on conical Radon transform, we allow a general weight depending on the distance of the integration point from the vertex. As first main result, we show uniqueness of inversion for that transform. To stably invert the weighted conical Radon transform, we use general convex variational regularization. We present numerical minimization schemes based on the Chambolle-Pock primal dual algorithm. Within this framework, we compare various regularization terms, including non-negativity constraints, $H^1$-regularization and total variation regularization. Compared to standard quadratic Tikhonov regularization, TV-regularization is demonstrated to significantly increase the reconstruction quality from conical Radon data.



\end{abstract}

\section{Introduction}
\label{sec:intr}

In this paper, we study the problem of reconstructing  a function
from its weighted conical Radon transform
\begin{equation*}
\RC f (\pp,\psi)
= \int_{\cone(\pp,\psi)} f(x)\,\U{\norm{x-\pp}}\ds (x) \,.
\end{equation*}
Here $\cone(\pp,\psi) \subseteq \R^n$ is a conical surface with
vertex $\pp $ on the unit sphere $\sph^{n-1}$, semi-axis
$-\pp$ and half-opening angle $\psi\in [0, \pi/2]$,
$U$ is a radial weight function,  and $\ds $ the standard
surface measure.

There are various variants of conical Radon transforms that differ in  the
restrictions of  the opening angle, the axis  direction or  the locations of the vertices of the conical surfaces, and  might include different weights.
Many  of such versions  of  conical Radon transforms have been   studied in the literature, see    \cite{allmaras2013,Maxim2009,palamodov2017reconstruction,ambartsoumian2016numerical,CreBon94,Hal14a,BasZenGul98,gouia2014exact,JunMoo15,Par00,Smi05,kuchment2016three,TomHir02,maxim2014filtered,MorEtAl10,terzioglu2018compton,moon2017analytic,smith2011line}.
For the special case of a  weight $U(r) = r^m$ with $m\in \Z$,  our  transform $\RC f$
has  been studied in \cite{schiefeneder2017radon}  in general dimension, and in \cite{haltmeier2017inversion} for the weight
$U(r)=e^{-\la r}$  in dimension   $n=2$.  Some  other works (for example \cite{kuchment2016three,moon2017analytic,palamodov2017reconstruction})
consider weights of the form   $U(r) = r^m$  for overdetermined conical Radon transforms.
The ill-posed character  is microlocally analyzed in \cite{zhang2018artifacts} in dimension  $n=2$ for the case $U(r)=1$, where also artifacts are characterized.  The  use of a general weight  for the conical Radon transform has not  been previously studied for any  variant of the conical Radon transform. 

The inversion of conical Radon transforms arises in single photon emission tomography (SPECT) using Compton
cameras~\cite{rogers2004compton,TodNigEve74}, as well as in other  applications such as single
scattering optical tomography~\cite{florescu2011inversion}. In the context of SPECT with Compton
cameras, the general radial weight that we consider in  this  paper allows to include the physically relevant effect of attenuation of photons.
In the context of SPECT using  the classical  Anger cameras where the exponential and the attenuated Radon transform appear this issue is well investigated; see for example \cite{bellini79compensation,novikov2002inversion,hawkins1988circular,inouye1989image,puro2001cormack,tretiak80exponential,markoe1985elementary,clough1983attenuated,natterer2001inversion,kunyansky2001new}.
For the conical Radon transform, attenuation is investigated much less and we are only aware of the works \cite{gouia2017,haltmeier2017}
considering the special case $U(r)=e^{-\la r}$.

The main contributions of this article are as follows. First, we  derive a  uniqueness result for $\RC$ extending the result of \cite{schiefeneder2017radon} to general weights.
 Second, in  order to address the ill-posedness of inversion of the weighted conical Radon transform, we apply  convex  variational  regularization
 \begin{equation*}
\tik_{\g, \alpha}(\f)  \coloneqq
\frac{1}{2} \norm{\RC \f - \g}_{L^2}^2
+ \alpha \reg( \f ) \to \min_{\f}  \,.
\end{equation*}
Here $\alpha$ denotes a positive regularization parameter and $\reg$ is a convex regularizer that allows to
include positivity as well as total variation (TV) regularization.
 For the implementation of the numerical minimization, we apply the  primal dual algorithm of   Chambolle and Pock \cite{chambolle2011}. To the best of our knowledge, neither  general convex variational regularization nor the
 Chambolle-Pock algorithm have previously been studied for the  conical Radon
 transform.  We present numerical studies comparing various regularizers.
 The presented  results clearly suggest  that TV-regularization outperforms the other regularizers
 for the considered phantom type.

\subsection*{Outline}

This paper is organized as follows.  In Section~\ref{sec:cone}, we define the weighted conical Radon
transform  and study its mathematical properties. In Section~\ref{sec:convex}, we consider  its variational
regularization. Additionally, we derive the numerical minimization algorithm and present numerical
results for $L^2$-regularization, $H^1$-regularization  and TV-regularization, all with and without
positivity constraint. The paper concludes  with a short summary and outlook presented in
Section~\ref{sec:conclusion}.

\section{The conical Radon transform}
\label{sec:cone}

Throughout this paper,  $U\in C^\infty([0,\infty),\R)$ denotes a  given function
with  $U(r) \geq 0$ for  all $r\geq 0$.
For  $\pp \in \sph^{n-1}$  and $\psi\in [0, \pi/2]$, we denote by
\begin{equation} \label{eq:cone}
	\cone(\pp,\psi)
	 =\{ \pp + r  \n \mid r \geq 0
	\text{ and }
	\n \in \sph^{n-1}
	\text{ with }  - \inner{\n}{\pp} =
	\cos\psi \}
\end{equation}
the surface  of a right circular half cone in $\R^n$ with
vertex $\pp$, central axis   $-\pp$,
and   half opening angle $\psi$.
We write  $\xx \coloneqq C^\infty_0(B_1(0))$
for the space of smooth functions  with compact support in
the   ball $B_1(0)$, and $\yy \coloneqq C^\infty(\sph^{n-1}\times [0,\pi/2])$.
The results in this section extend the results
of~\cite{schiefeneder2017radon}  for the conical Radon transform with 
weights $U(r) = r^m$ to the case of general radial weights.

\subsection{Definition and basic properties}

We first define the transform studied in this paper.

\begin{definition}[Weighted conical Radon transform]
For  $\f \in \xx$ we  define   the  weighted conical Radon transform
(with vertices on the sphere, orthogonal axis and weighting function $U$) by
\begin{equation}\label{eq:crt1}
\RC \f   \colon \sph^{n-1}  \times [0, \pi/2] \to \R
\colon
  (\pp, \psi) \mapsto   \int_{\cone(\pp,\psi)} f(x)\,\U{\norm{x-\pp}}\ds (x)  \,.
\end{equation}
\end{definition}
The weighted conical Radon transform integrates a function
$\f \in \xx$ supported inside the unit  ball  over cones
with vertices on the unit sphere $\sph^{n-1} = \set{x \in \R^n \mid  \norm{x} =1}$ and central axis orthogonal to $\sph^{n-1}$ pointing to the interior
of the sphere. Alternative expressions for the conical Radon transform
that will be helpful for our analysis are derived next.

\begin{proposition}\label{prop:R}
Suppose $f\in \xx$.
\begin{enumerate}
\item\label{prop:R-1} If $\Qm \in O(n)$  and $\pp \in \sph^{n-1}$, then  $(\RC \f)(\Qm\pp,\edot) = \RC(f \circ \Qm)(\pp,\edot)$.

\item \label{prop:R-2}
For every $(\pp, \psi) \in \sph^{n-1} \times [0, \pi/2]$, we have
\begin{align}
(\RC\f)(e_1, \psi )   \label{eq:Rf1}
      &= \int_0^\infty  \U{r} (r\sin (\psi))^{n-2}
  \\ & \nonumber \hspace{0.05\textwidth}
    \times \int_{\sph^{n-2}}  f  \kl{ 1- r \cos(\psi), r \sin(\psi) \eta } \ds(\eta)  \dd r \,,
\\ (\RC\f)(e_1,\psi) \label{eq:Rf2}
	&=
	\int_{0}^{\pi-2\psi} U\kl{\frac{\sin(\al)}{\sin(\psi+\al)}} \frac{(\sin(\psi))^{n-1}
    (\sin(\al))^{n-2}}{(\sin(\psi+\al))^{n} }
  \\ & \nonumber \hspace{0.05\textwidth}
    \times \int_{\sph^{n-2}}
    \f \left( \frac{\sin(\psi)}{\sin(\psi+\al)} (\cos(\al), \sin(\al) \eta )\right)
     \ds (\eta) \dd \al \,.
\end{align}
\end{enumerate}
\end{proposition}

\begin{proof} The proof is similar as the proof of~\cite[Lemma 2.2]{schiefeneder2017radon}  with only a slight modification concerning the weight function.
\end{proof}

\subsection{Adjoint  transform}

For $g\in \yy$ we define the weighted conical  backprojection  $\RC^\sharp g  \colon  \R^n  \to \R$ by
\begin{multline}\label{eq:adj}
 \RC^\sharp g(x)  \coloneqq
 \int_{\sph^{n-1}} \int_0^{\pi/2} g(z,\psi)\, U\kl{\norm{x-z}}
 \\
 \times \delta\bkl{\inner{(x-z)}{z+\norm{x-z}\cos(\psi)}}\,\dd \psi\, \ds (z) \,,
\end{multline}
if $x \in B_{\RR}(0)$, and $\RC^\sharp g(x)=0$ otherwise.
In the next proposition, we will show that    $\RC^\sharp$ is the
(formal)   adjoint of $\RC$ with respect to the $L^2$-inner products
\begin{align}
\ip{f_1}{f_2}_{L^2}
& \coloneqq
\int_{\R^n} f_1 (x)  f_2(x) \,\dd x
\\
\ip{g_1}{g_2}_{L^2}
& \coloneqq
\int_{\sph^{n-1}} \int_0^{\pi/2} g_1(z,\psi) g_2(z,\psi)\,  \dd \psi\,\ds(z)
\end{align}
on $\xx$ and  $\yy$,  respectively.

\begin{proposition}[Weighted conical  backprojection]\label{prop:adj}
The operator  $\RC^\sharp$ is  the  $L^2$-adjoint of
$\RC$, that is, for all $f \in \xx $ and  $g \in  \yy$ it holds
$\ip{\RC  f}{g}_{L^2} = \ip{f}{ \RC^\sharp   g}_{L^2}$.
\end{proposition}

\begin{proof}
For every $f\in \xx$ and $g\in \yy$ we have
\begin{align*}
\ip{\RC \f}{g}_{L^2}
&=
\int_{\sph^{n-1}} \int_0^{\pi/2}g(z,\psi) \bkl{\RC\f}(z,\psi)\,  \dd \psi\,\ds(z)
\\
&=
\int_{\sph^{n-1}} \int_0^{\pi/2} \int_{\cone(\pp,\psi)} g(z,\psi) f(x) \U{\norm{x-\pp}}
\ds(x)\, \dd \psi\,\ds(z)
\\
&=\int_{\sph^{n-1}} \int_0^{\pi/2} \int_{\R^n} g(z,\psi) f(x)\,\U{\norm{x-z}}
\\
&\hspace{0.2\textwidth}
\times
\delta\bkl{\inner{(x-z)}{z}+ \norm{x-z}\cos(\psi)}\, \dd x\, \dd \psi\,\ds(z)
\\
&=\int_{\R^n} \int_{\sph^{n-1}} \int_0^{\pi/2}  g(z,\psi)\,\U{\norm{x-z}}\,\\
&\hspace{0.2\textwidth}
\times
\delta\bkl{\inner{(x-z)}{z}+ \norm{x-z}\cos(\psi)}\,f(x)\,   \dd \psi\,\ds(z)\,\dd x
\\
&=\int_{\R^n}   f(x)  \bkl{\RC^\sharp g}(x) \,\dd x
\\
&=
\ip{f}{\RC^\sharp g}_{L^2}
\,,
\end{align*}
where we used the Theorem of Fubini.
\end{proof}

\subsection{Decomposition in one-dimensional integral equations}

Next we derive an explicit decomposition of the conical Radon transform  in one-dimensional integral operators   (see Theorem~\ref{thm:glk2}).
For that  purpose  we  use the spherical harmonic
decompositions
\begin{align} \label{eq:fexp}
\f(r\theta)
	&=
	\sum_{\ell=0}^{\infty}
	\sum_{k=1}^{N(n,\ell)}
	\f_{\lk}( r)\,Y_{\lk} (\theta ) \,,
	\\ \label{eq:gexp}
    (\RC \f)(\pp,\psi)
	&=
	\sum_{\ell=0}^{\infty}
	\sum_{k=1}^{N(n,\ell)}
	(\RC \f)_{\lk} (\psi) Y_{\lk}(\pp) \,.
\end{align}
Here $Y_{\lk}$, for $\ell\in\N$ and $k \in \set{1, \dots,   N(n,\ell)} $,  denote spherical harmonics \cite{Mue66,seeley66}  of degree $\ell$ forming a  complete orthonormal system in $ \sph^{n-1}$.
The set of all $(\ell, k)$ with $\ell\in\N$ and $k \in \set{1, \dots,   N(n,\ell)}$ will be denoted by $I(n)$.
Let $C^\mu_\ell$ denote the Gegenbauer polynomials normalized in such  a way that $C^\mu_\ell(1)=1$.

As in~\cite{schiefeneder2017radon}, we derive  different relations between $f_{\lk}$ and $(\RC\f)_{\lk}$ in the form of Abelian integral equations.

\begin{theorem}[Generalized Abel equation \label{thm:glk2} for $\f_\lk$]
Let   $\f \in \xx$
and let $\f_{\lk}$ and $(\RC\f)_{\lk}$ be as \eqref{eq:fexp} and \eqref{eq:gexp} for  $ (\ell,k) \in I(n)$. Then, for  $\psi \in [0, \pi/2]$,
\begin{equation}\label{eq:glk2}
	(\RC\f)_{\lk}(\psi)
	=\abs{\sph^{n-2}}
	\int_{\sin(\psi)}^{1} f_{\lk}(\rho)
	\frac{  \rho\, \K_\ell(\psi,\rho)}{\sqrt{\rho^2-(\sin(\psi))^2}}\dd \rho \,,
\end{equation}
with the kernel functions
\begin{align} \nonumber
& \K_\ell(\psi,\rho) \coloneqq
\bkl{\sin(\psi)}^{n-1} \sum_{\sigma = \pm 1}
\U{\cos(\psi)-\sigma\,\sqrt{\rho^2-\sin(\psi)^2}}
\\ \nonumber
& \hspace{0.2\textwidth} \times
\bkl{\cos(\psi)-\sigma\,\sqrt{\rho^2-\sin(\psi)^2}}^{n-2}
\\ \label{eq:Kell}
&	\hspace{0.2\textwidth} \times \,\Cn\left(\bkl{\sin(\psi)^2+\sigma\,\cos(\psi)\sqrt{\rho^2-\sin(\psi)^2}}/\rho\right)    \,.
\end{align}
\end{theorem}

\begin{proof}
Following the proof of~\cite[Lemma 3.1]{schiefeneder2017radon}, we obtain
\begin{multline} \label{eq:glk}
	\forall \psi \in [0, \pi/2] \colon \quad
	(\RC\f)_{\lk}(\psi)
		=\abs{\sph^{n-2}}
	\int_0^{\pi-2\psi}
	f_{\lk}\left(\frac{\sin(\psi)}{\sin(\psi+\al)}\right)\U{\frac{\sin(\alpha)}{\sin(\psi+\al)}}
	\\
	\times \frac{(\sin(\psi))^{n-1}(\sin(\al))^{n-2}}{(\sin(\psi+\al))^{n}}
	\Cn(\cos(\al))\dd \alpha \,.
\end{multline}
Splitting the integral  in one  integral
over $ \alpha <  \pi/2 - \psi $ and  one over $\al \geq \pi/2 - \psi$ and proceeding similar as in~\cite[Theorem 3.2]{schiefeneder2017radon} yields the claim.
\end{proof}

\begin{proposition}
Let   $\f \in \xx$\label{prop:glk3}
and let $\f_{\lk}$ and $(\RC\f)_{\lk}$ be as \eqref{eq:fexp} and \eqref{eq:gexp}.
Further, for every   $ (\ell,k) \in I(n)$  denote
\begin{enumerate}[label=(\alph*)]
\item\label{it:T1}
 $\hat{\g}_{\lk}(t)\coloneqq \abs{\sph^{n-2}}^{-1} (1-t)^{-(n-2)/2}
(\RC\f)_{\lk}(\arccos \sqrt{t})$;
\item \label{it:T2} $\hat{\f}_{\lk}(s)\coloneqq   f_{\lk}\left( \sqrt{1-s} \, \right) / 2$;
\item \label{it:T3} $\F_\ell(t,s) \coloneqq \sum_{\sigma = \pm 1} \sigma^\ell\, \U{\sqrt{t}- \sigma \sqrt{t-s} \, }
\bkl{\sqrt{t}- \sigma \sqrt{t-s} \, }^{n-2}$ \\
$ \times \Cn \left(\frac{\sqrt{t}\sqrt{t-s}+ \sigma (1-t)}{\sqrt{1-s}}\right) $.
\end{enumerate}
Then $\hat{\f}_{\lk}$ and $\hat{\g}_{\lk}$ are related via:
\begin{align}\label{eq:glk3}
	\forall t \in [0,1] \colon \quad
	\hat{\g}_{\lk}(t)=\int_{0}^{t} \hat{\f}_{\lk}\left(s\right)
	\frac{\F_\ell(t,s)}{\sqrt{t-s}}\dd s \,.
\end{align}
\end{proposition}

\begin{proof}
Follows the lines of~\cite[Lemma 3.3]{schiefeneder2017radon}.
\end{proof}

\subsection{Uniqueness of reconstruction}

In this section, we show uniqueness of recovering the function by its conical radon transform and thus the injectivity of $\RC$ by showing solution uniqueness of the  Abelian integral equations in Theorem~\ref{thm:uni}.
Since the kernel function of the Abelian integral equation~\eqref{eq:glk2} has zeros on the diagonal, the proof of the uniqueness relies on the  uniqueness result derived in \cite{schiefeneder2017radon}, which we briefly state at this point:
For $a,b\in\R$ with $a<b$ we set $\dom{a,b}\coloneqq \left\{(t,s)\in\R^2\mid a\leq s\leq t\leq b\right\}$.
\begin{lemma}
Suppose\label{lem:abel} that $ F\colon \dom{a,b}\to \R$, where $a < b$, satisfies the following:
\begin{enumerate}[label=(F\arabic*)]
	\item\label{lem:abel-1} $F\in C^3(\dom{a,b})$.
	\item\label{lem:abel-2} $N_F  \coloneqq  \set{ s \in [a,b) \mid F(s,s)=0}$
	is finite and consists of simple roots.
	\item\label{lem:abel-3} For every $s \in N_F$, the gradient $(\beta_1,\beta_2)
	\coloneqq \nabla F (s,s)$  satisfies
	\begin{equation}\label{eq:ineq:Abel}
		1+\frac{1}{2}\,\frac{\beta_1}{\beta_1+\beta_2}>0 \,.
	\end{equation}
\end{enumerate}
Then, for any $g \in C([a,b])$, the integral equation
$\forall t \in [a,b] \colon
	\int_a^t \frac{F(t,s)}{\sqrt{t-s}}\,\f(s)\dd s =g(t) $
has at most one solution $f \in C([a,b])$.
\end{lemma}

\begin{proof}
See ~\cite[Theorem~3.4]{schiefeneder2017radon}.
\end{proof}

\begin{theorem}[Uniqueness \label{thm:uni}  of recovering $\f_\lk$] Suppose  the function $U$ is real-analytic and satisfies $\frac{n+1}{2}+\frac{\sqrt{s}\,U'(\sqrt{s})}{\U{\sqrt{s}}}>0$ for all $s\in[0,2]$. For any  $\f \in \xx$ and any $(\ell,k) \in I(n)$, the spherical harmonic coefficient $\f_{\lk}$ of $\f$   can be recovered as the unique solution of
\begin{equation*}
	\forall \psi \in [0, \pi/2]\colon \quad
	(\RC\f)_{\lk}(\psi)
	=\abs{\sph^{n-2}}
	\int_{\sin(\psi)}^{1} \f_\lk (\rho)
	\frac{ \rho \, \K_\ell(\psi,\rho) \dd \rho}{\sqrt{\rho^2-(\sin(\psi))^2}} \,,
\end{equation*}
with the kernel functions $\K_\ell $ defined by \eqref{eq:Kell}.
\end{theorem}

\begin{proof} We proceed as in the proof of Theorem 3.5 in~\cite{schiefeneder2017radon}.
Let $f \in \xx$   vanish outside a ball of radius $1- a^2$, where $a\in(0,1)$.
 We show that equation \eqref{eq:glk3} has a unique
solution, which is sufficient according to Lemma \ref{prop:glk3}.
For that purpose, we verify  that
$\F_\ell \colon  \dom{a, 1} \to \R$  satisfies the conditions
\ref{lem:abel-1}-\ref{lem:abel-3} in Lemma~\ref{lem:abel}.

\begin{itemize}[wide]
\item Ad \ref{lem:abel-1}:
Using the abbreviations $U_n$ for the function $r\mapsto U(r)\,r^{n-2}$  and  $C \coloneqq  \Cn$, the kernel
$\F_\ell$ can be written in the form
\begin{equation*}
\F_\ell(t,s)  =  \sum_{\sigma = \pm 1} \sigma^\ell
\Un{\sqrt{t}- \sigma \sqrt{t-s} \, } C \left(\frac{\sqrt{t}\sqrt{t-s}+ \sigma (1-t)}{\sqrt{1-s}}\right) \,.\end{equation*}
The function  $\F_\ell$ is clearly smooth on $\set{(t,s) \in \dom{a, 1} \mid t \neq s}$. Using  $C(-x) = (-1)^\ell C(x)$ and that $U_n$ is analytic, $\F_\ell$ is an even real analytic function in $\sqrt{t-s}$ which shows that  $\F_\ell$ is also smooth on the diagonal $\set{(t,s) \in \dom{a, 1} \mid t = s}$.

\item Ad \ref{lem:abel-2}:
Let $ v(s)   \coloneqq \F_\ell(s,s)
=2  \,\Un{\sqrt{s}}C (\sqrt{1-s})$ denote the restriction of the kernel to the diagonal.

As an orthogonal polynomial, the function $C$ has only a finite number of isolated and simple roots.    We conclude the same holds true for the function $v$, using that $U$ is a positive function.

\item Ad \ref{lem:abel-3}:
Let $s_0 \in [a, 1)$ be a zero of $v$. Setting  $(\beta_1,\beta_2) \coloneqq \nabla \F_\ell (s_0,s_0)$, we obtain
 \begin{equation}\label{eq:beta12}
	\beta_1+\beta_2
	=
	v'(s_0)=
	- \frac{\Un{\sqrt{s_0}}}{\sqrt{1-s_0}}
	C'\left(\sqrt{1-s_0}\right) \,.
\end{equation}
Following the Proof of Theorem 3.5 in \cite{schiefeneder2017radon} one verifies
\begin{equation}\label{eq:beta1}
\beta_1 = \frac{1}{ \sqrt{1-s_0}}C'(\sqrt{1-s_0}) \,\bkl{(n-3)\Un{\sqrt{s_0}}-2\sqrt{s_0}U_n'\bkl{\sqrt{s_0}}}.
\end{equation}
From \eqref{eq:beta12} and  \eqref{eq:beta1} and using the definition of $U_n$ it follows that
\begin{equation*}
1+ \frac{\beta_1}{2(\beta_1 + \beta_2)}
= \frac{n+1}{2}+ \frac{\sqrt{s}\,U'(\sqrt{s})}{\U{\sqrt{s}}} \,.
\end{equation*}
Thus, by assumption, property \ref{lem:abel-3} is satisfied.
\end{itemize}
Lemma \ref{lem:abel} now yields that
$\hat \f_\lk$ is the unique  solution of   \eqref{prop:glk3}.
\end{proof}

\section{Convex  regularization and iterative minimization}
\label{sec:convex}

The stability analysis made in   \cite{haltmeier2017inversion} (performed there in 2D)
shows   that the inversion of the conical Radon transform is highly ill-conditioned
(compare  \cite{zhang2018artifacts}, where the artifacts are characterized in 2D).
To address the  ill-posedness, in our previous works \cite{haltmeier2017inversion,schiefeneder2017radon}
we used
standard quadratic  Tikhonov regularization. In this paper,
we  go one step further and apply variational  regularization allowing
a general convex regularization term.
For the following we set  $\XX \coloneqq  L^2(B_1(0))$
and $\YY \coloneqq L^2(\sph^{n-1}\times [0,\pi/2])$.

\subsection{Variational regularization}

In   order to stably  solve  $\RC  \f  = \g$,
we consider  variational regularization which
consists  in minimizing the generalized  Tikhonov
functional
\begin{equation} \label{eq:tik}
\tik_{\g, \alpha}(\f)  \coloneqq
\frac{1}{2} \norm{\RC \f - \g}_{L^2}^2
+ \alpha \reg( \f ) \,.
\end{equation}
Here $\alpha$ is some non-negative constant and $\reg \colon \XX  \to [0, \infty]$ is a
proper  convex, coercive and weakly lower semi-continuous
functional.
Tikhonov regularization is used to stably
approximate $\reg$-minimizing solutions of $\RC( \f) = \g$, which
are defined as elements
in $\argmin \set{  \reg( \f ) \mid \f \in \XX
\wedge \RC( \f) = \g }$.

We first  derive the  boundedness  of $\f \mapsto \RC  \f$  with respect to the  $L^2$-norm.

\begin{lemma}\label{lem:Rcont}
Suppose $r \mapsto r^{n - 3} U(r)^2$ is integrable over $(0,2)$.
For  some constant $c \in (0, \infty)$ we have
$\norm{\RC \f}_{L^2}  \leq c  \norm{\f}_{L^2}$ for all $f \in \xx$.
In particular, $\RC$ can be uniquely extended
to a bounded operator $\RC \colon \XX \to \YY$, which  has
bounded adjoint.
\end{lemma}

\begin{proof}
Let  $\pp \in \sph^{n-1}$ and $\Qm \in O(n)$ satisfy  $\Qm e_1 = \pp$.
By Proposition \ref{prop:R}, we have
 \begin{align*}
	\lVert &(\RC \f)(\pp, \edot ) \rVert^2_{L^2}
     \\
    & = \int_{0}^{\pi/2}
          \abs{\RC(\f \circ \Qm )(e_1, \psi )}^2 \dd \psi \\
     &=  \int_{0}^{\pi/2} (\sin (\psi))^{2(n-2)}
     \\
    &  \hspace{0.1\textwidth}
    \times  \biggl\lvert \int_0^2
    \int_{\sph^{n-2}}
    U(r) r^{n-2}
    (\f \circ \Qm)  \left(1 -  r \cos(\psi), r \sin(\psi) \eta \right)
\ds(\eta)  \dd r \biggr\rvert^2 \dd \psi
 \\ &\leq
 \abs{\sph^{n-2}}
\biggl(\int_0^2  r^{n - 3} U(r)^2 \dd r \biggr)
\biggl( \int_{0}^{\pi/2} (\sin (\psi))^{2(n-2)}
\\
 &  \hspace{0.1\textwidth}
  \times \int_0^2 \int_{\sph^{n-2}}
r^{n-1 } \abs{(\f \circ \Qm)  \left(1 -  r \cos(\psi), r \sin(\psi) \eta \right) }^2
\ds(\eta)  \dd r  \dd \psi \biggr)\,.
 \\ &\leq
 \abs{\sph^{n-2}}
 \biggl(\int_0^2  r^{n - 3} U(r)^2 \dd r \biggr)
 \norm{f}_{L^2}^2
 \,.
\end{align*}
Integration over  $\pp \in \sph^{n-1}$ and using that
$\int_0^2  r^{n - 3} U(r)^2 \dd r$ is finite yields the claimed
estimate.  Because  $\xx \subseteq \XX$ is dense, the
operator can  be extended  to a bounded operator on $\XX$
in a unique way.  The adjoint is therefore bounded, too.
\end{proof}

The continuity  of  $\RC$   together with  standard results for convex variational regularization
implies the following.

\begin{theorem}[Existence,  stability and convergence]\label{thm:tik}
Let $\reg \colon \XX  \to [0, \infty]$ be
proper, convex, coercive and weakly lower semi-continuous,
and suppose $r \mapsto r^{n - 3} U(r)^2$ is integrable over $(0,2)$.
Then the following hold:

\begin{enumerate}
\item
For every $\g \in \YY$  and $\al >0$,
$\tik_{\g, \alpha}$ has at least one minimizer.

\item
Let $\al>0$, $\g  \in \YY$, and let $(\g_k)_{k \in \N} \in \YY^\N$
satisfy   $\norm{\g - \g_k}_{L^2} \to 0$. Then,
every sequence $\f_k \in \argmin \tik_{\g_k, \alpha}$
 has a weakly convergent subsequence $(\f_{\tau(k)})_{k \in \N}$, and its limit $\f$
  is a minimizer of $\tik_{\g, \alpha}$ with  $\reg(\f_{\tau(k)}) \to
\reg(\f)$ for $k \to \infty$.

\item
Let $\g \in \ran(\RC)$, $\kl{\delta_k}_{k\in \N} \in (0, \infty)^\N$
converge to zero, and let $(\g_k)_{k\in \N} \in \YY^\N$ satisfy  $\norm{\g - \g_k} \leq \delta_k$.
Suppose further that $(\alpha_k)_{k\in \N} \in (0, \infty)^\N$
satisfies   $\alpha_k \to 0$ and $\delta_k^2/\alpha_k \to  0$ as $k \to \infty$,
and choose $\f_k \in  \argmin \tik_{\g_k,\alpha_k}$ for every $k \in \N$.
\begin{itemize}[topsep=0em,itemsep=0em]
\item $(\f_k)_{k\in \N}$  has a weakly converging subsequence.
\item
The limit  $\f$  of every weakly  convergent subsequence
$(\f_{\tau(k)})_{\ell\in \N}$  of $(\f_k)_{k\in \N}$ is a $\reg$-minimizing solution
of  $\RC(\f) = \g$ and  satisfies $\reg(\f_{\tau(k)}) \to \reg(\f)$.

\item If the $\reg$-minimizing solution of $\RC( \f) = \g$
is unique, then $\f_k \rightharpoonup \f$.
\end{itemize}
\end{enumerate}
\end{theorem}

\begin{proof}
Follows from the boundedness $\RC$ derived in Proposition \ref{lem:Rcont} together with general results for  convex variational  regularization~\cite{scherzer2009variational}.
\end{proof}

Note that the uniqueness of a $\reg$-minimizing solution of $\RC( \f) = \g$
 is guaranteed if   its solution is unique.  Uniqueness  of the
  $\reg$-minimizing solution is also guaranteed in the case where $\reg$ is strictly convex. This property is satisfied,  for example, for standard Tikhonov regularization
 where $\reg(\f) = \norm{\f}_{L^2}^2$,  or for $\ell^q$-regularization where
 $\reg(\f) = \sum_{\la \in \La} \abs{\ip{\f_\la}{\varphi_\la}}^q$ with $q>1$ and some 
 frame  $(\varphi_\la)_\la$ of $\XX$.

\subsection{Iterative minimization}

In the numerical results presented below, we
consider  the following instances for the regularizer   $\reg \colon \XX  \to [0, \infty]$:
\begin{itemize}
	\item
	$\reg(\f) = \frac{1}{2} \norm{\f}^2_{L^2}$
    ($L^2$-regularization);	
	
	\item
	$\reg(\f) = \frac{1}{2} \norm{\Do \f}^2_{L^2}$
    ($H^1$-regularization);
	
	\item
	$\reg(\f) =   \norm{ \Do \f}_{L^1}$
    (TV-regularization).
\end{itemize}
Additionally, we consider any of these methods with an added
convex constraint. All resulting regularization functionals are proper, convex,
coercive and weakly lower semi-continuous.

For numerically minimizing the Tikhonov functional  \eqref{eq:tik},
we consider its discrete counterparts. For that purpose, let
$\fnum \in (\R^{N+1})^{\otimes n}$ be the discrete phantom,
$\gnum^\delta \in \R^{P \times (Q+1)} $ discrete data and
$\Cnum \colon (\R^{N+1})^{\otimes n} \to \R^{P \times (Q+1)}
$  the discretization of the forward operator. For all considered
regularizers, we can write the resulting discrete Tikhonov functional
in the form
    \begin{equation} \label{eq:tikN}
	\Tnum (\fnum)
	\coloneqq \frac{1}{2} \norm{ \Cnum  \fnum -  \gnum^\delta}_2^2
	+ \frac{\al}{q} \norm{   \Lnum \fnum  }_q^q +  I_\cset (\fnum)\,.
\end{equation}
Here
  $I_\cset $ denotes the indicator of some convex  set $\cset \subseteq  (\R^{N+1})^{\otimes n}$
 defined by $I_\cset(\fnum) = 0 $ if  $\fnum \in \cset$ and $I_\cset(\fnum) =\infty $ else.
In particular, in the case that we take
$\cset = ([0, \infty)^{N+1})^{\otimes n}$  it guarantees non-negativity.
The mapping  $\Lnum \in \set{\Dnum, \Inum}$ stands either  for the discrete gradient
$\Dnum \colon  (\R^{N+1})^{\otimes n} \to  ((\R^{N+1})^{\otimes n} )^2$
 in the case of $H^1$-regularization and TV-regularization, or for the
identity  operator $\Inum$  on $(\R^{N+1})^{\otimes n}$
 in the case  of $L^2$-regularization. The parameter $q\in \set{1, 2}$ is
 taken $q=1$ in the case  of TV regularization, and $q=2$ in the cases of
 $L^2$-regularization
 or $H^1$-regularization.

\begin{algorithm}
\caption{Chambolle-Pock\label{alg:cp} Algorithm for minimizing the functional \eqref{eq:tikN} in case of constrained  TV regularization}
\label{alglstv}
\begin{algorithmic}[1]
\STATE
Choose $a \leq \|( \Cnum,  \Dnum)\|_2$;
$\tau \gets 1/a$;
$\sigma \gets1/a$;
$\theta \gets 1$;
$ k \gets 0$
\STATE Initialize $\fnum_0$, $\pnum_0$, and $\qnum_0$ to zero values
\STATE $\unum_0 \gets \fnum_0$
\WHILE{stopping criteria not satisfied}
\STATE $\pnum_{k+1} \gets (\pnum_k + \sigma( \Cnum \unum_k - \gnum^\delta))/(1+\sigma)$
\STATE $\qnum_{k+1} \gets \alpha (\qnum_k  + \sigma \Dnum \unum_k )/
\max \set{ \alpha \mathbf{1},|\qnum_k + \sigma \Dnum \unum_k |}$
\STATE $\fnum_{k+1} \gets  \Pnum_\cset(\fnum_k - \tau \Cnum^* \pnum_{k+1} - \tau \Dnum^*   \, \qnum_{k+1}) $
\STATE $\unum_{k+1} \gets \fnum_{k+1} + \theta(\fnum_{k+1} - \fnum_k)$
\STATE $k \gets k+1$
\ENDWHILE
\end{algorithmic}
\end{algorithm}

The Tikhonov functional~\eqref{eq:tikN} can  be minimized  by
various convex optimization methods~\cite{CombPes11}.
In this work we use the minimization algorithm  of \cite{sidky2012convex},
which is a special instance of the  Chambolle-Pock
primal-dual algorithm \cite{chambolle2011}.
It can be applied to any instance of regularization functional that we consider in this paper.
 For TV-minimization with convex constraint, the resulting algorithm is  summarized
in  Algorithm~\ref{alg:cp}. Here $\norm{\edot}_2$ denotes the matrix norm induced by the Euclidean norm, $\mathbf{1}$ is the matrix with all entries set to $1$ and $\Pnum_\cset$ denotes the projection on the convex set $\cset$.
In the case of $L^2$-regularization and
$H^1$-regularization, the algorithms look similar, only   line 6 has to be  replaced  by
the update rule
\begin{equation*}
\qnum_{k+1} \gets \frac{ \alpha (\qnum_k + \sigma \Lnum \unum_k) }{\alpha +  \sigma}\,.
\end{equation*}
For motivation and a derivation of  Algorithm~\ref{alg:cp} as well as a convergence analysis
we refer to the original  papers
\cite{chambolle2011, sidky2012convex}.

\subsection{Numerical simulations}

For the presented numerical results, we assume weight $U(r) = e^{-\mu r }$
and consider the two dimensional case. In this case,
the weighted conical Radon transform reduces to the attenuated V-line transform 
\begin{multline} \label{eq:Vtrafo}
(\RC \f) (\ph, \psi)
=
\\ \sum_{\sigma = \pm 1}\int_0^\infty
\f \kl{ (\cos(\ph),\sin(\ph))
- r(\cos(\ph-\sigma\psi),\sin(\ph-\sigma\psi))  } e^{-\mu r } \dd r \,,
\end{multline}
where $\ph \in [0, 2\pi)$  and $\psi \in [0, \pi/2]$.
The $V$-line transform \eqref{eq:Vtrafo}
corresponds to transforms studied in
\cite{schiefeneder2017radon} and \cite{moon2017analytic}.

\begin{figure} \centering
\includegraphics[width=0.45\textwidth]{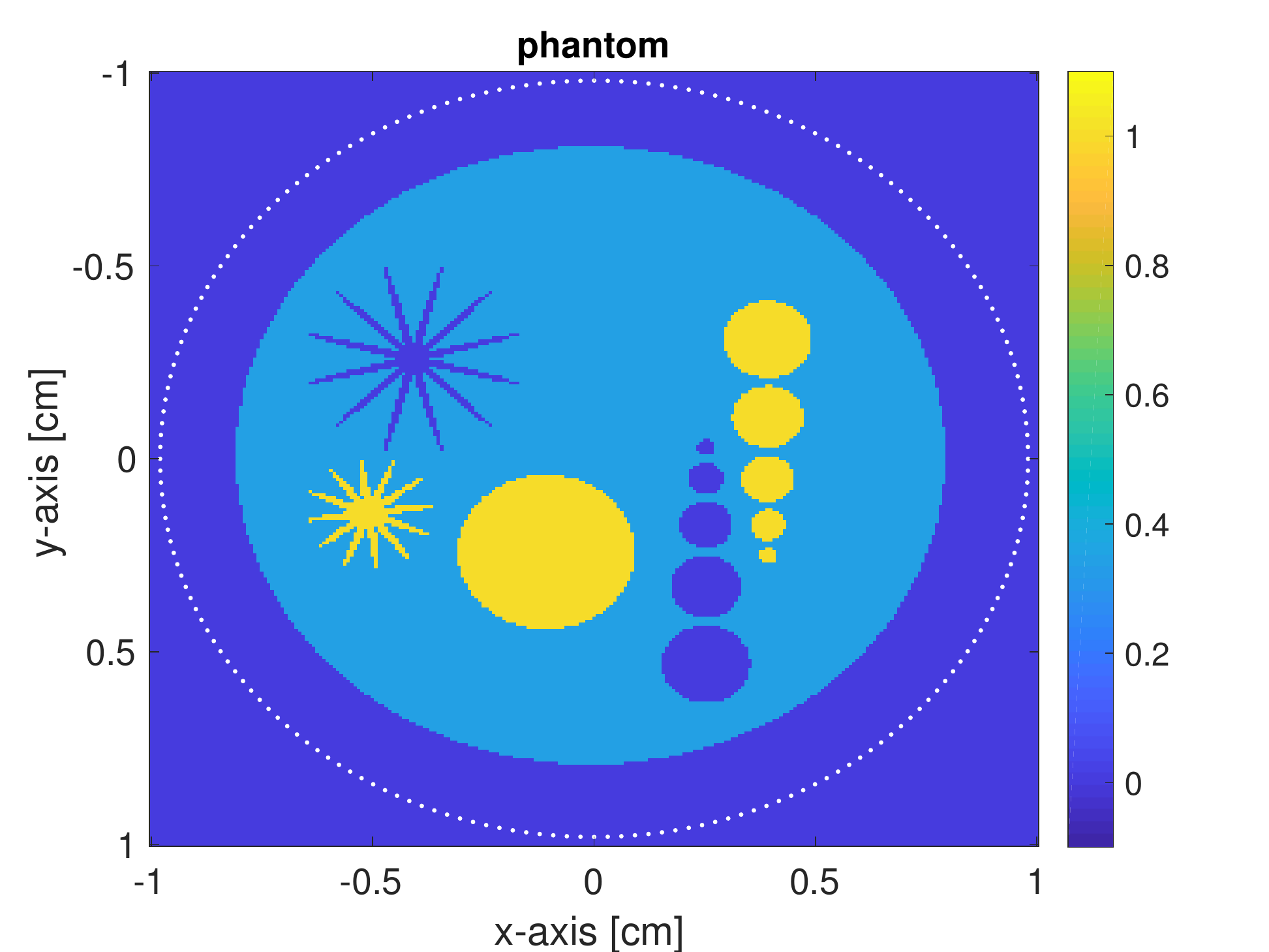}
\includegraphics[width=0.45\textwidth]{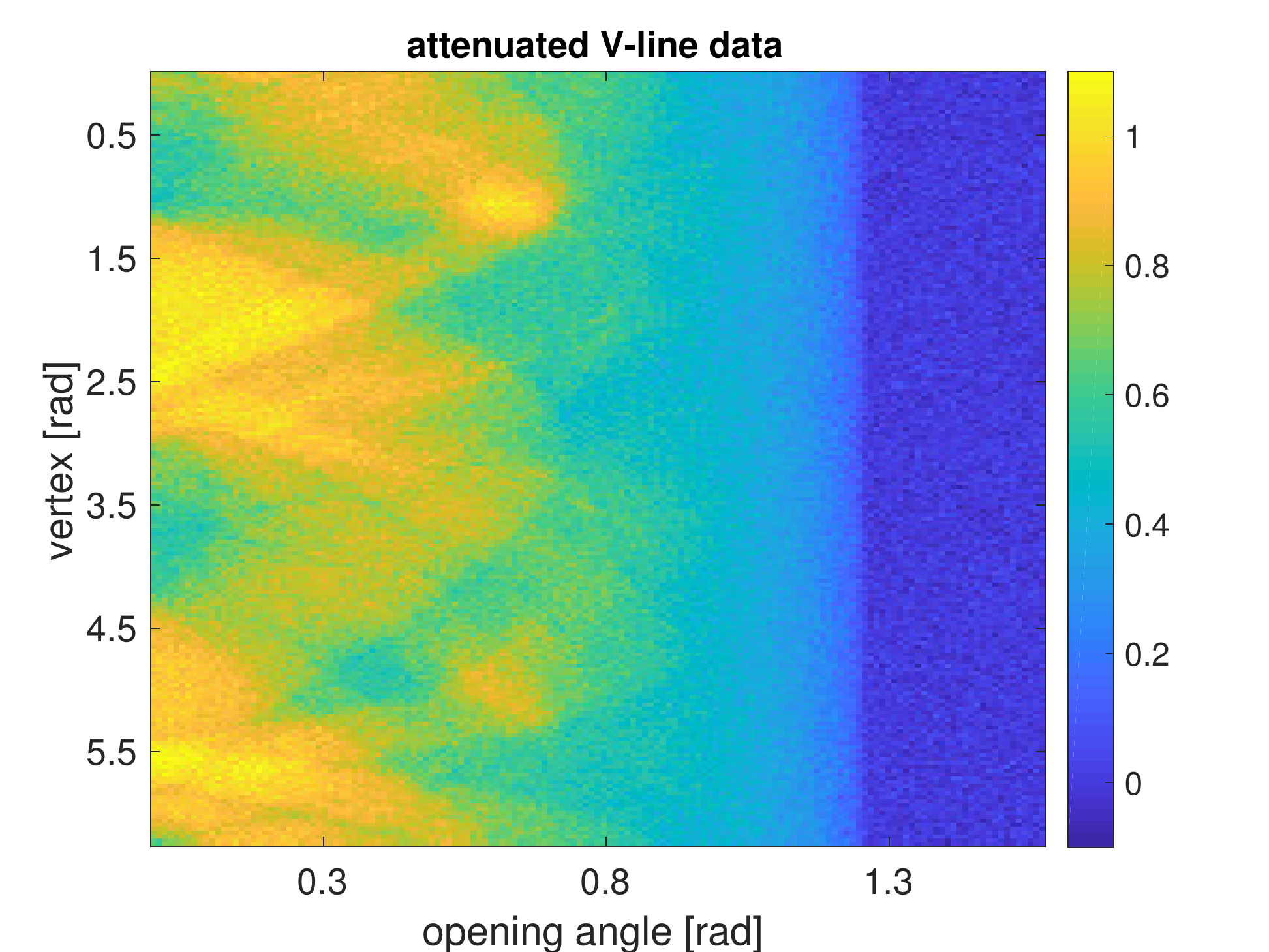}
\caption{\textbf{Phantom and noisy data used for the numerical results}.
Left: The numerical phantom $f \in \fnum \in \R^{257 \times 257}$
consists of the superposition of ellipse and two star shaped tumor
like   structures. Right: Simulated data $\gnum^\delta =  \Cnum \fnum + \noisenum \in \R^{200 \times 151}$
with $5\%$ noise added.}\label{fig:phant}
\end{figure}

\subsubsection*{Discrete forward and adjoint operator}

The  discrete forward and adjoint operators are defined
by considering the entries $\fnum[i_1,i_2] $  as sampled values of
a  function $f$  at locations
$(-1,-1) +   2(i_1, i_2)/N$ for $(i_1, i_2) \in \set{0, \dots, N}^2$
and  replacing the function $f$ by a bilinear interpolant $T[\fnum]$.
The attenuated V-line transform $\Cnum  \fnum = \gnum$ is  discretized by  numerically computing
the integral of the interpolant $T[\fnum]$  over each of the two branches of the V-lines with the composite
trapezoidal rule. We take
  $(\cos(\ph[k]),\sin(\ph[k]))$  with  $\ph[k] = 2\pi (k-1)  / P $  for $k \in \set{1, \dots, P}$
for the vertex positions,  and  $\psi[\ell] = \pi \,  \ell/(2   Q) $ with $\ell \in \set{0, \dots, Q}$
for the opening angles. For the numerical integration of each branch,
we use $N+1$ equidistant radii in the interval $[0,2]$.
The discrete backprojection operator is defined by using the trapezoidal rule
similar to the well-known approach   for the classical Radon transform \cite{natterer2001}.

For the following numerical results we use $N=256$, $P=200$, $Q=  150$, 
and  $\mu =  0.5$.
The used discrete phantom $\fnum$ and the numerically computed   data
$\Cnum \fnum$ with added Gaussian  noise are shown in Figure~\ref{fig:phant}.

\begin{figure}\centering
\includegraphics[width=0.35\textwidth]{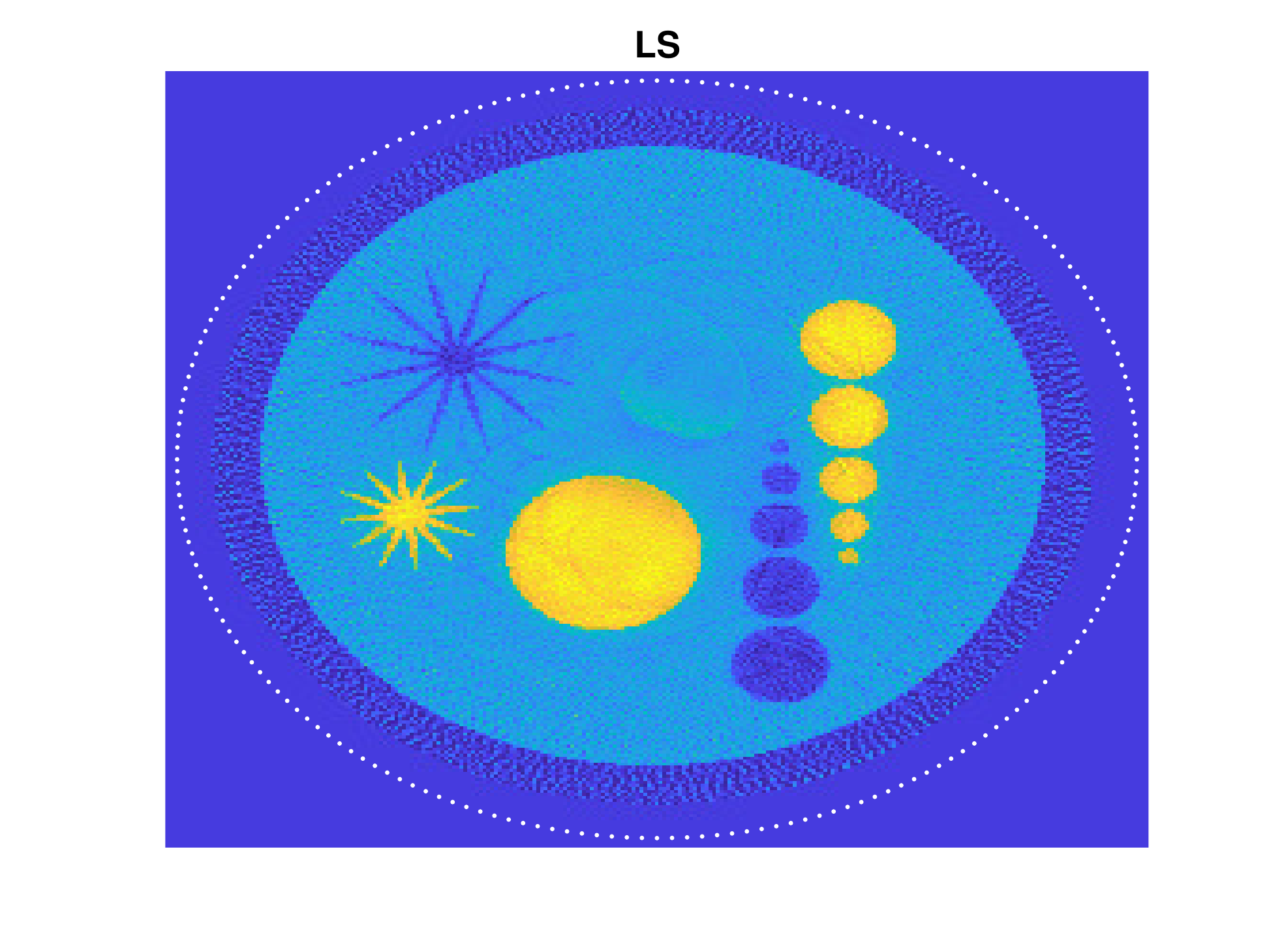}
\includegraphics[width=0.35\textwidth]{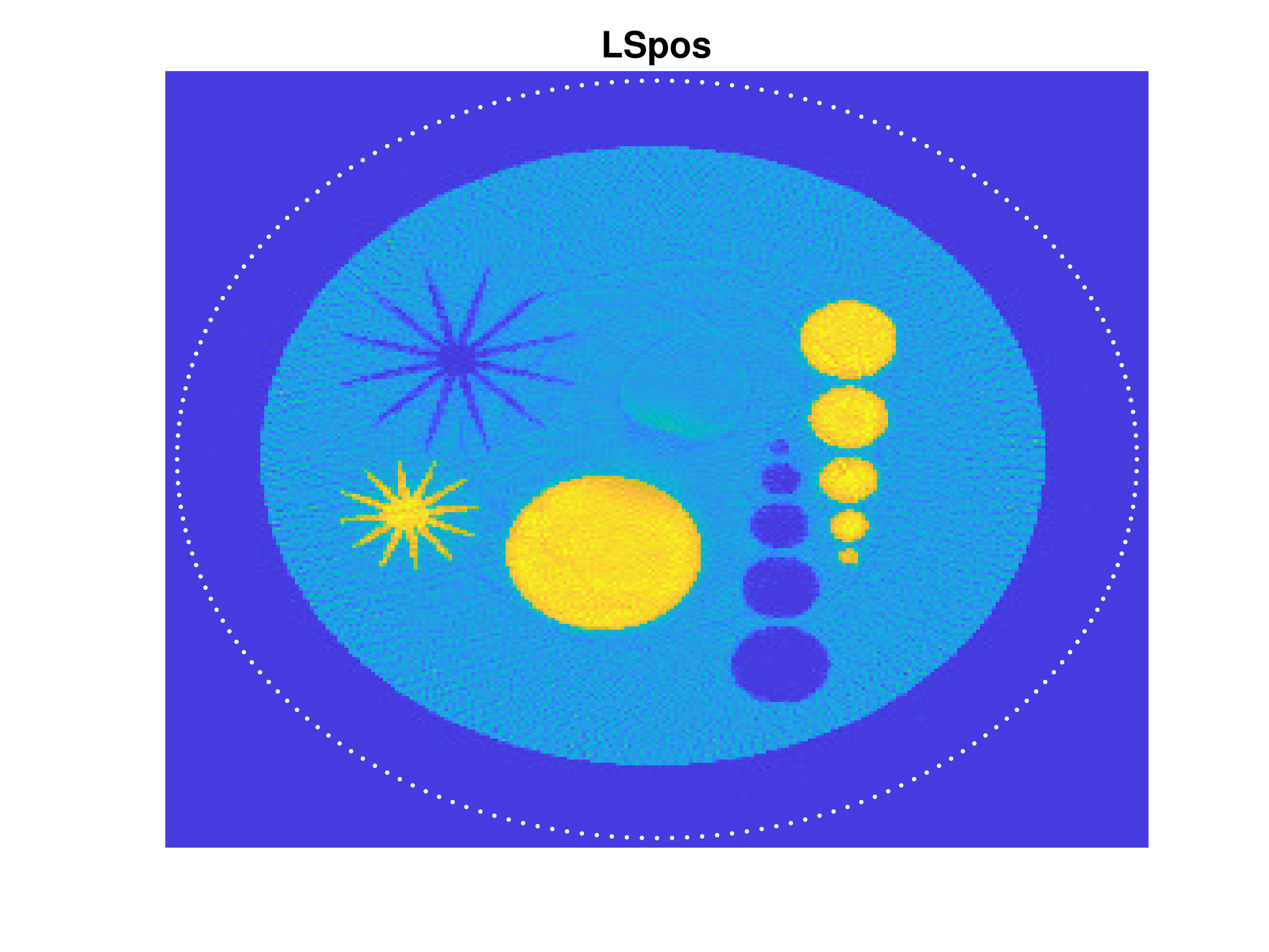}
\includegraphics[width=0.35\textwidth]{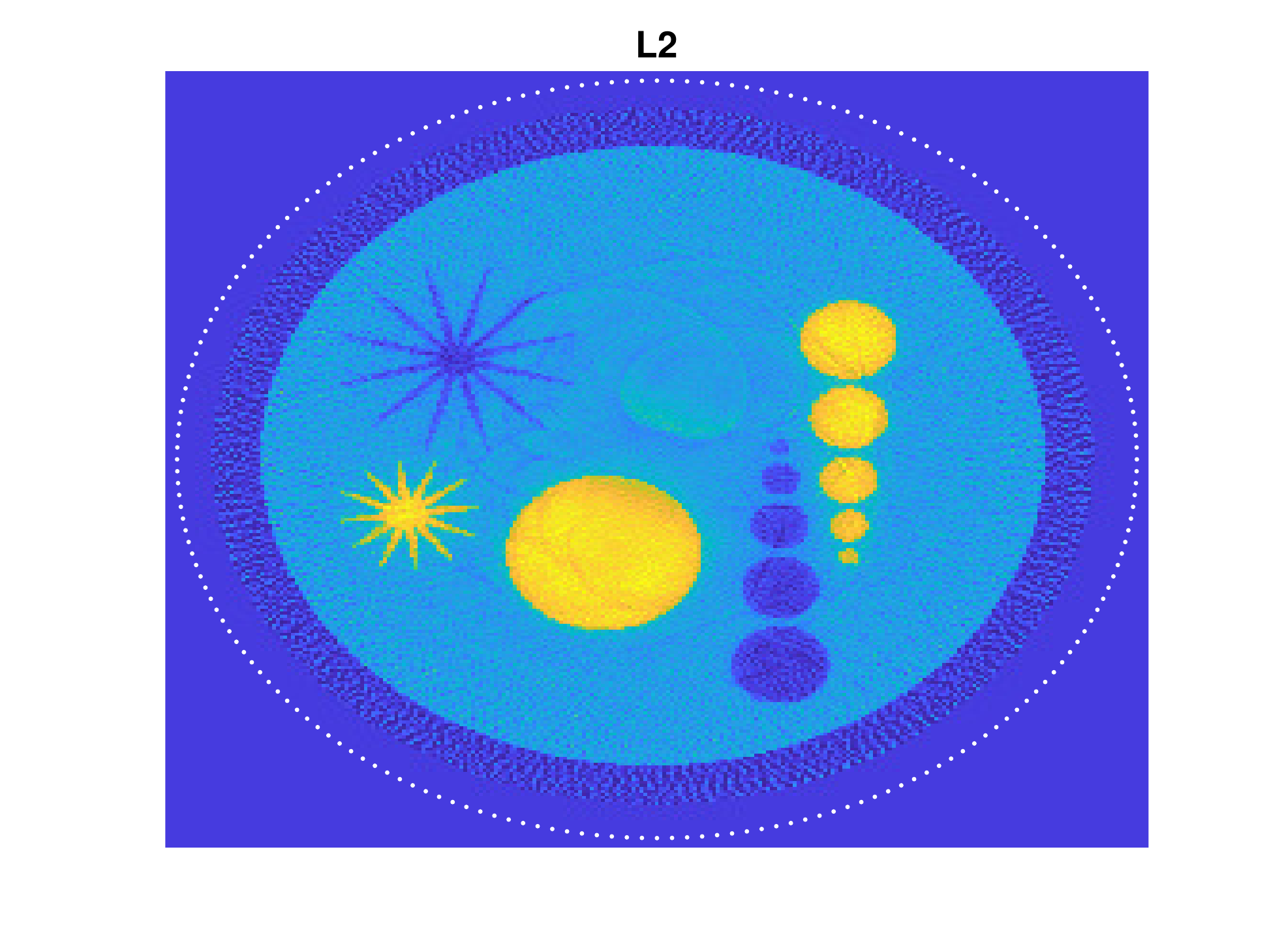}
\includegraphics[width=0.35\textwidth]{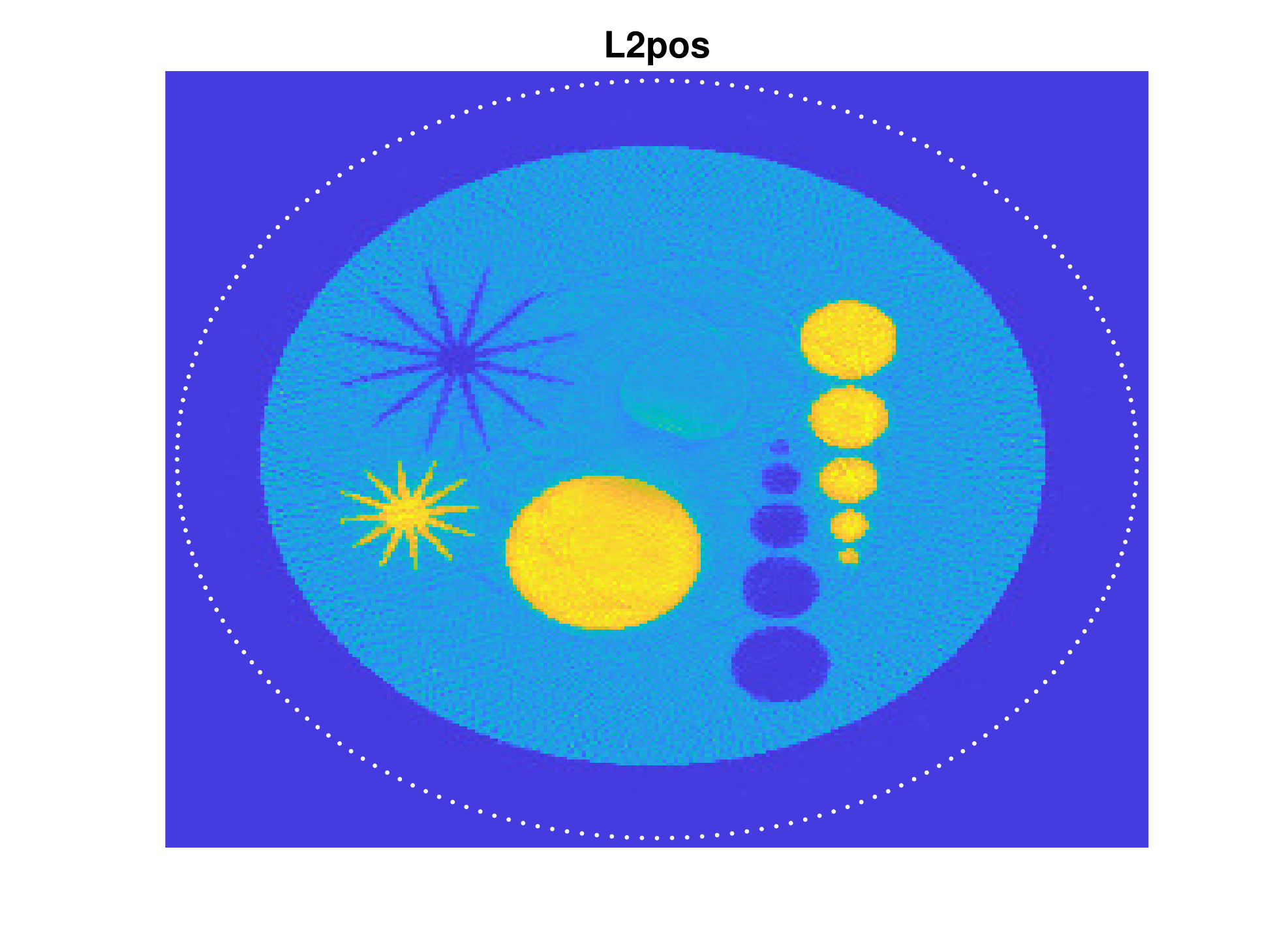}
\includegraphics[width=0.35\textwidth]{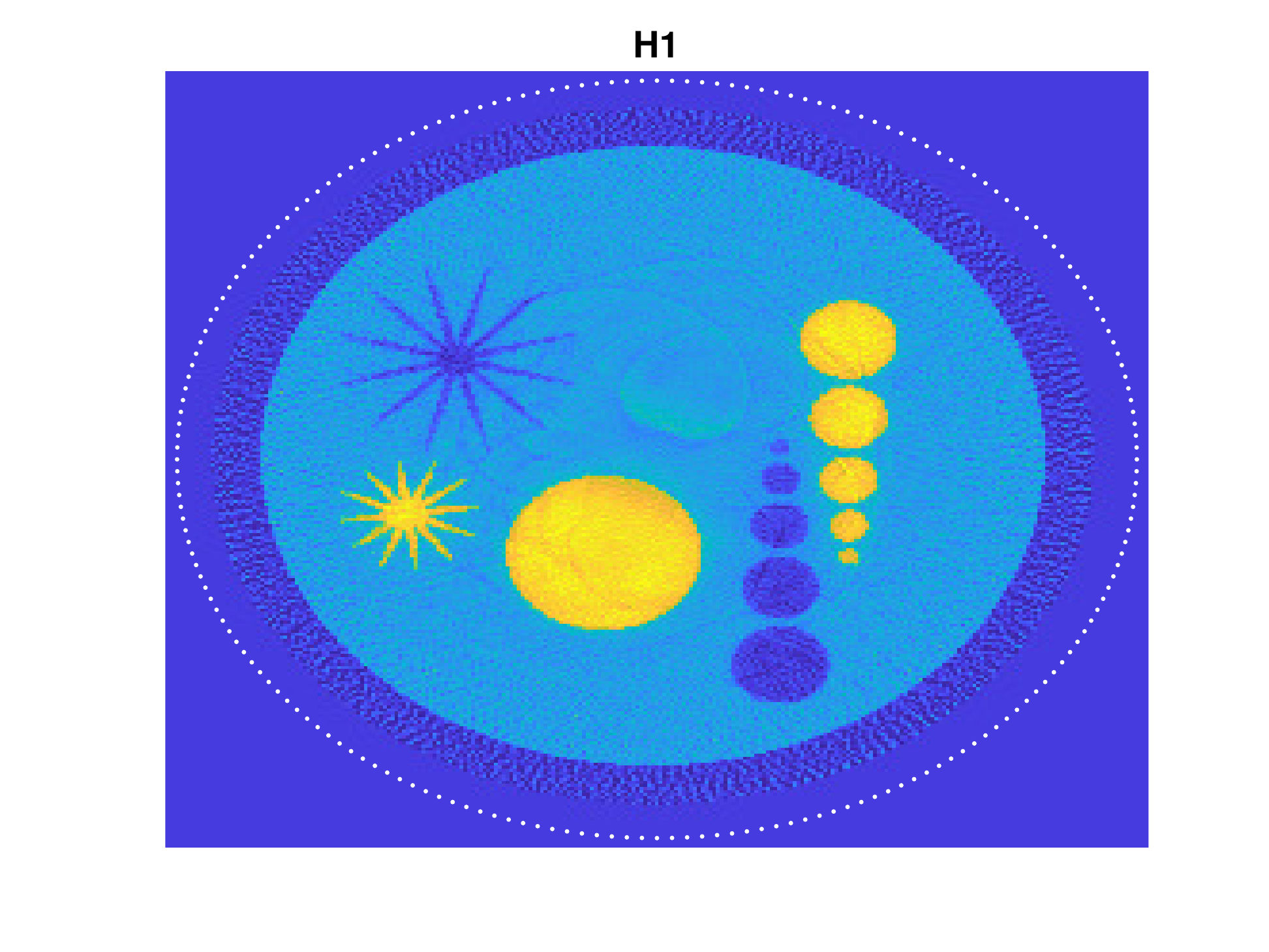}
\includegraphics[width=0.35\textwidth]{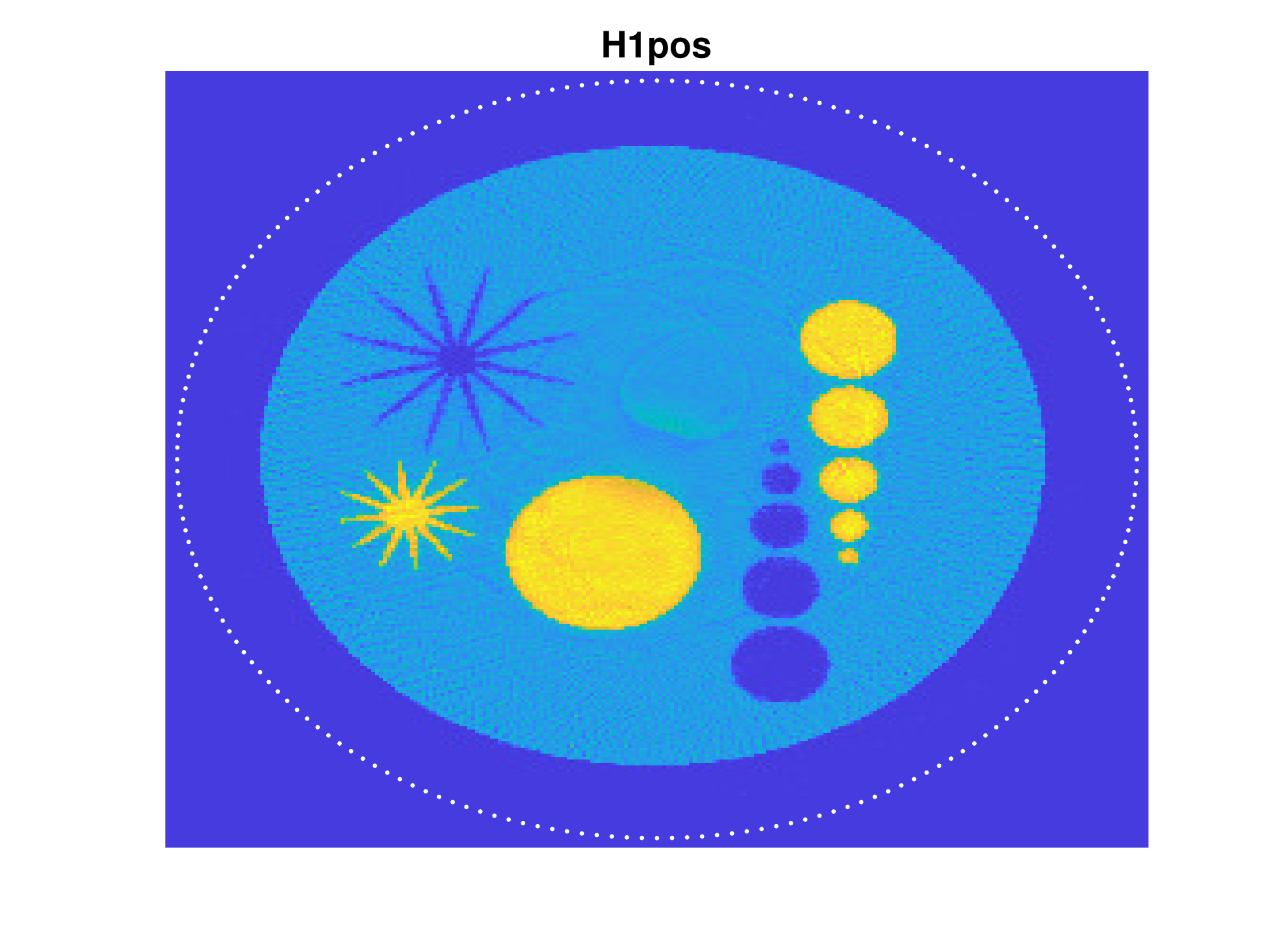}
\includegraphics[width=0.35\textwidth]{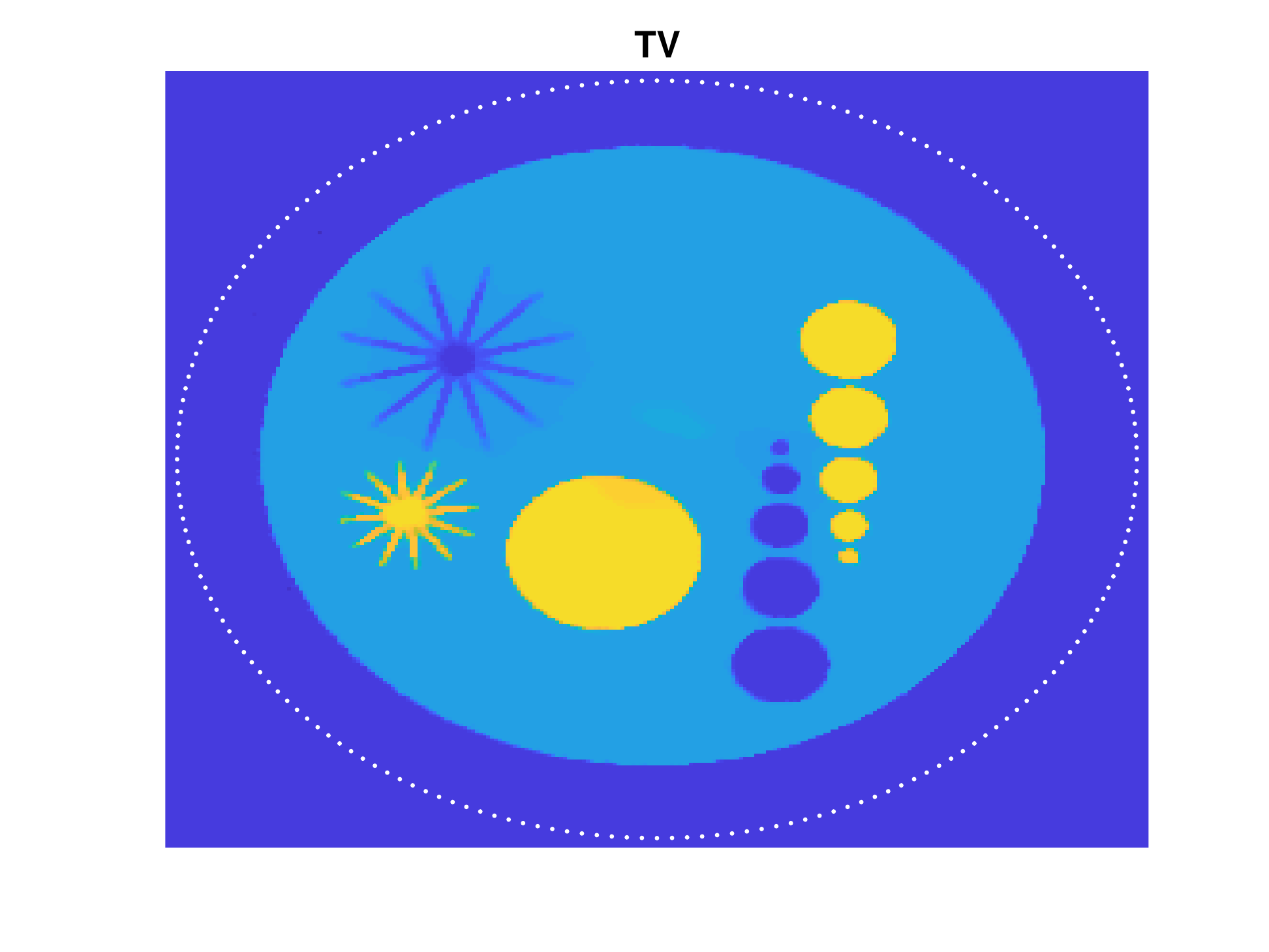}
\includegraphics[width=0.35\textwidth]{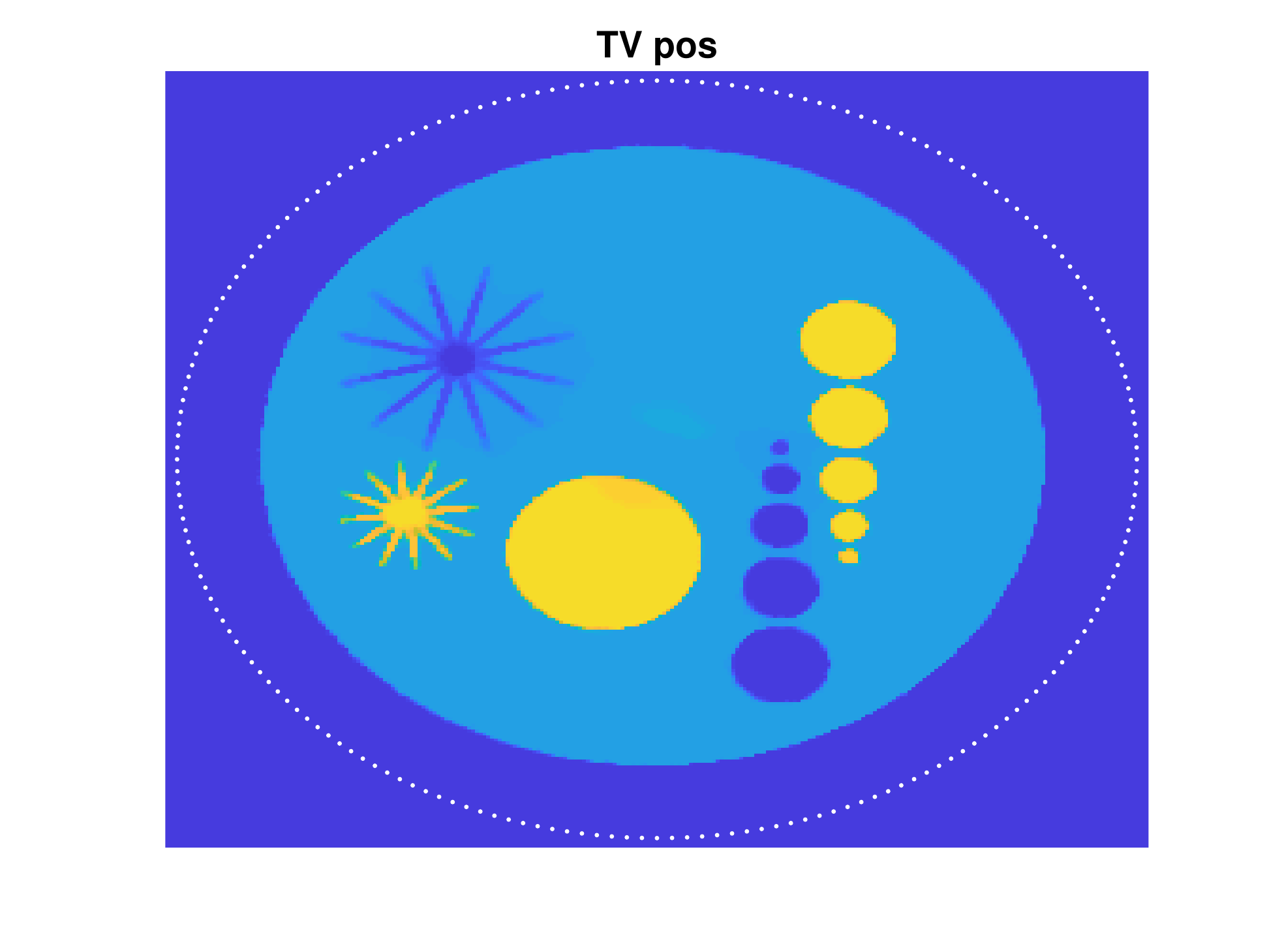}
\includegraphics[width=0.35\textwidth]{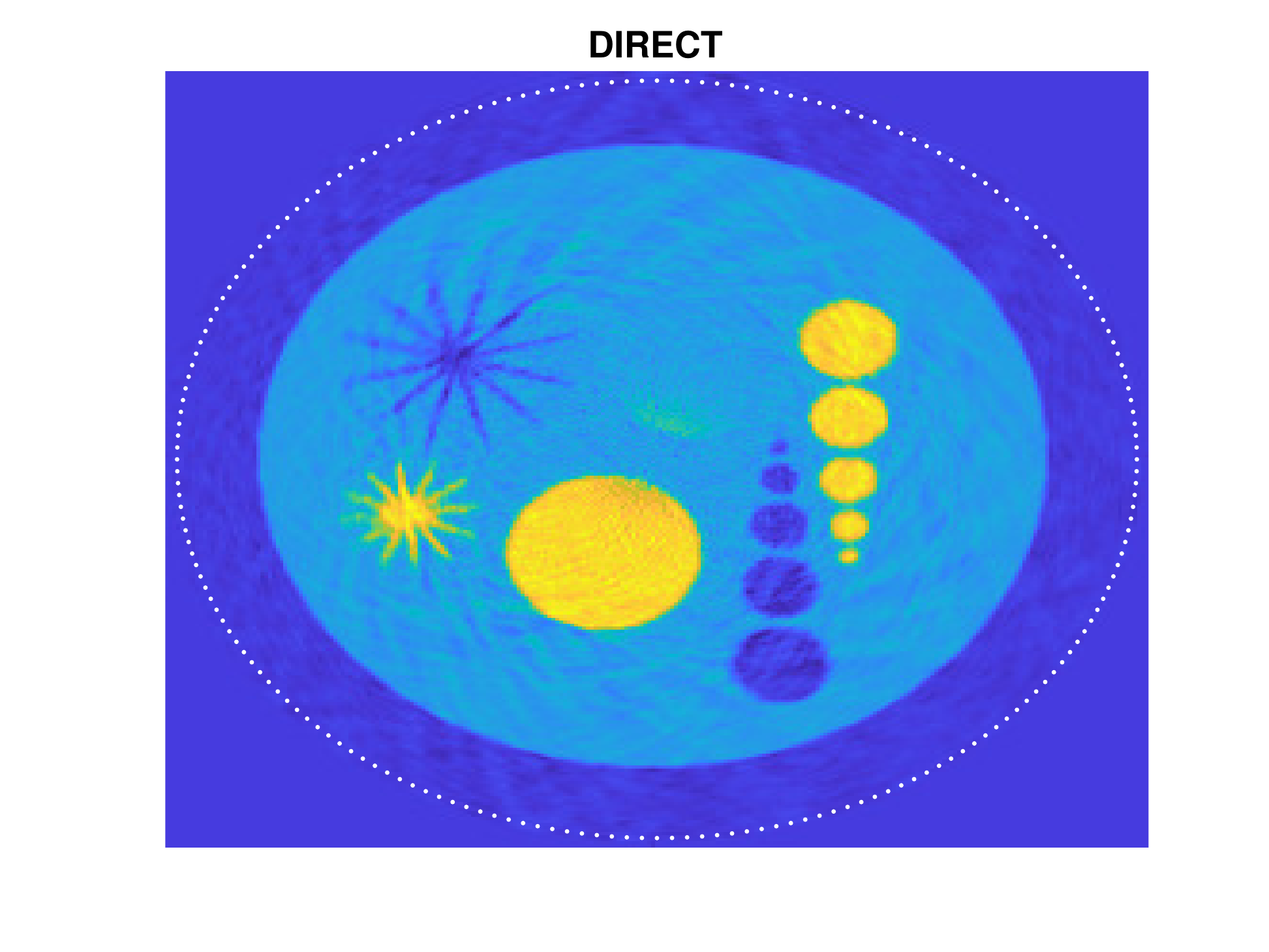}
\includegraphics[width=0.35\textwidth]{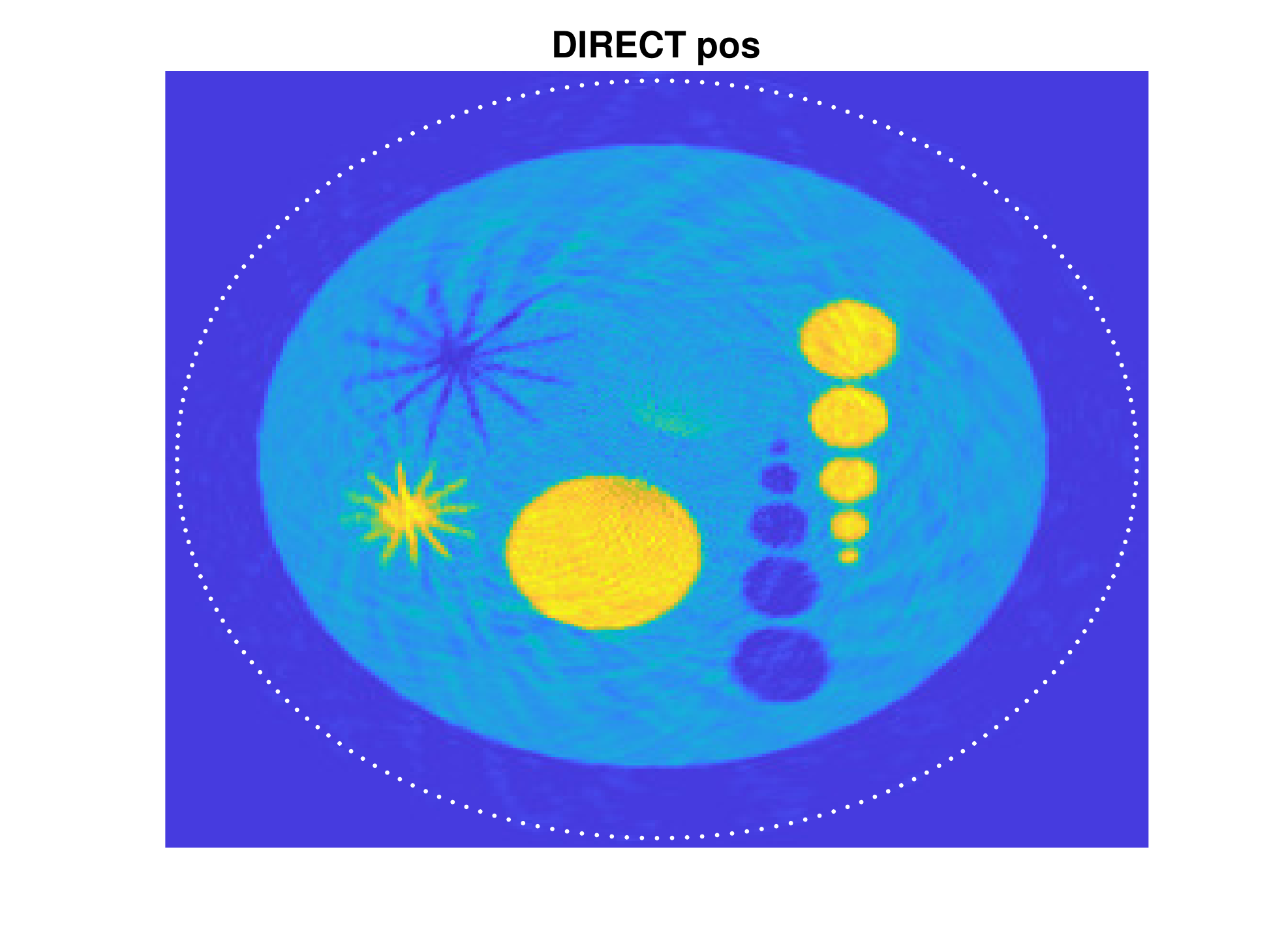}
\caption{\textbf{Reconstructions from exact data.}
Top left: least squares.
Top right: positivity constraint least squares.
Line 2, left: $L^2$-regularization.
Line 2, right: $L^2$-regularization with positivity constraint.
Line 3, left: $H^1$-regularization.
Line 3, right: $H^1$-regularization with positivity constraint.
Line 4, left: TV-regularization.
Line 4, right: TV-regularization with positivity constraint.
Bottom,  left: Fourier reconstruction from  \cite{haltmeier2017inversion}.
Bottom,  right: Nonnegative part of Fourier reconstruction.
}\label{fig:recE}
\end{figure}

\begin{figure}\centering
\includegraphics[width=0.45\textwidth]{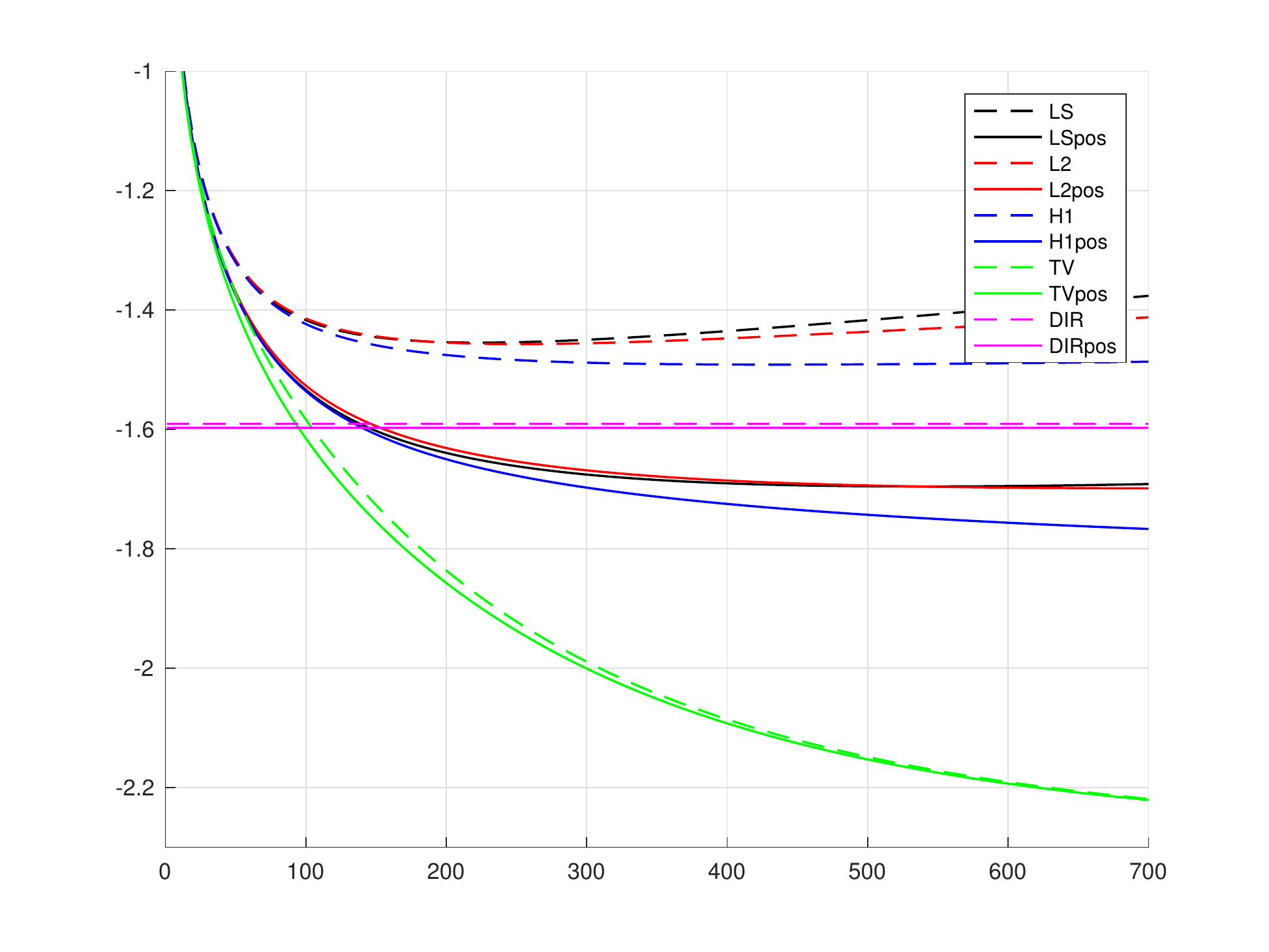}
\includegraphics[width=0.45\textwidth]{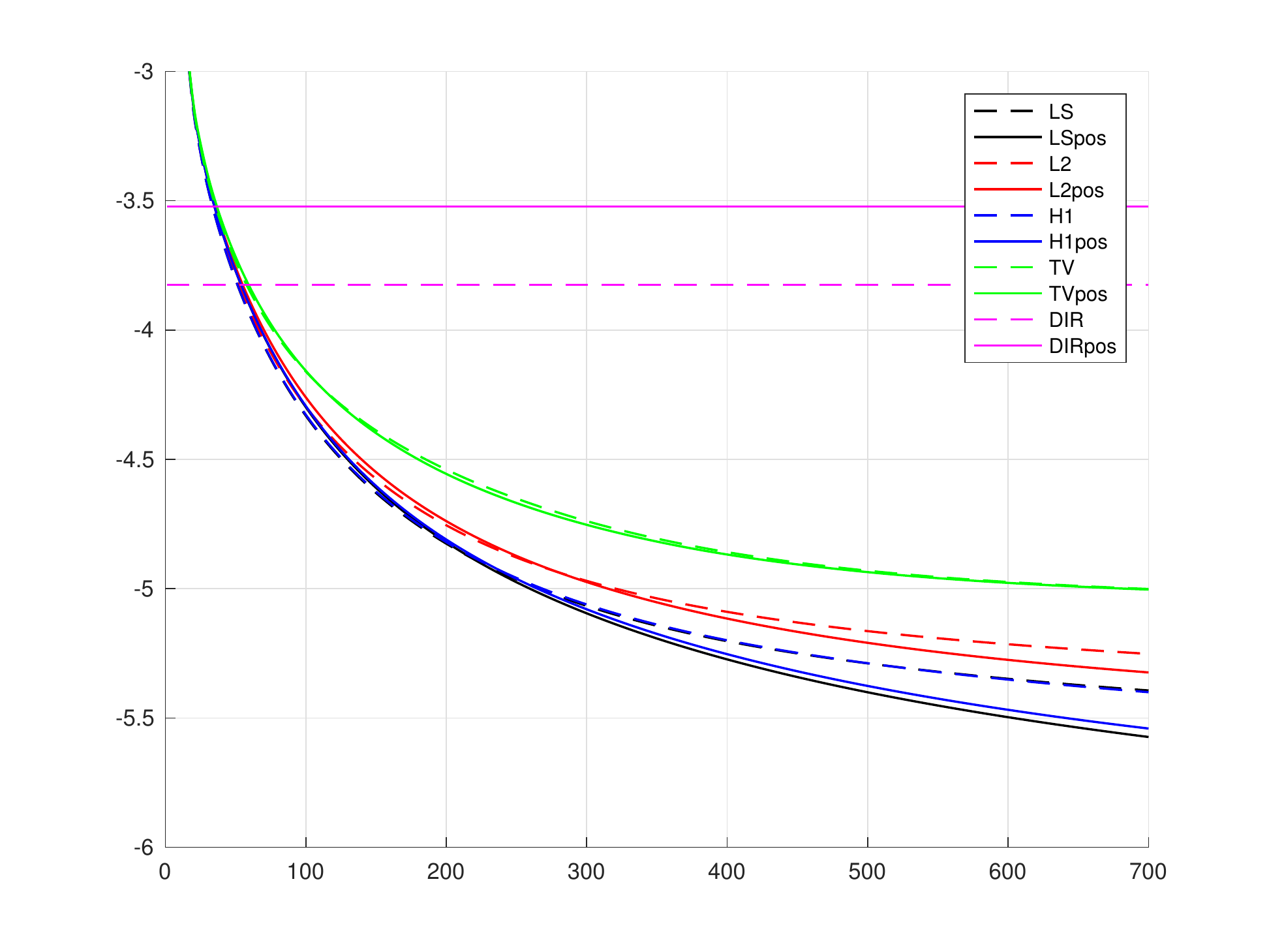}
\caption{\textbf{Reconstruction errors and residuals from exact data.}
Left: Logarithm $\log_{10} E^2(\fnum_j)$ of squared  reconstruction error.
Right: Logarithm $\log_{10} R^2(\fnum_j)$ of squared  residuals.
}\label{fig:errE}
\end{figure}

\subsubsection*{Reconstruction from exact data}

We first investigate the variational regularization methods on simulated
data without noise added. We consider  $L^2$-regularization,  $H^1$-regularization and TV-regularization.
All  methods are used  with and without positivity constraint. 
 For comparison purposes, we also  perform reconstructions using plain least squares,  positivity constrained least squares and the direct Fourier reconstruction method from  \cite{haltmeier2017inversion}. 
 The  regularization parameter has been taken $\alpha = 0.002$ for
 TV and $H^1$ regularization and   $\alpha = 0.01$ for $L^2$ regularization. 
  The reconstruction results  after
 700 iterations are  shown in Figure~\ref{fig:recE}. Total variation minimization
 (with and  without positivity constraint) clearly outperforms all other methods.
 In  particular, $L^2$- as well as $H^1$-regularization  contain a ghost source
 corresponding to the big disc, which is not contained in the TV-reconstruction.
Using the  least squares  methods, the ghost image is also  not that severe.

To quantitatively evaluate the reconstructions, we compute the
relative squared $\ell^2$-error and relative squared $\ell^2$-residuals, respectively,
\begin{align}
E^2(\fnum_j)
&\coloneqq
\frac{\sum_{i_1,i_2}\abs{\fnum_j[i_1,i_2]-\fnum[i_1,i_2]}^2 }{\sum_{i_1,i_2}\abs{\fnum[i_1,i_2]}^2}
\\
R^2(\fnum_j)
&\coloneqq
\frac{\sum_{k,\ell}\abs{\Cnum \fnum_j[k,\ell]-\gnum[k,\ell]}^2 }{\sum_{k,\ell}\abs{\gnum[k,\ell]}^2} \,.
\end{align}
Logarithmic plots  of $E^2(\fnum_j) $ and $R^2(\fnum_j) $ are shown in
Figure \ref{fig:errE}. The reconstruction error for the TV-regularization is much smaller than for all
other methods.
The residuals, as expected, are smallest for plain   least squares.
As a consequence of the ill-posedness, this  does not imply small reconstruction error.
In fact, the least squares methods  show  a slight semi-convergence behavior
which is due  data error introduced by the numerical implementation and the ill-posedness
of the inverting the attenuated V-line transform.

\begin{figure}\centering
\includegraphics[width=0.35\textwidth]{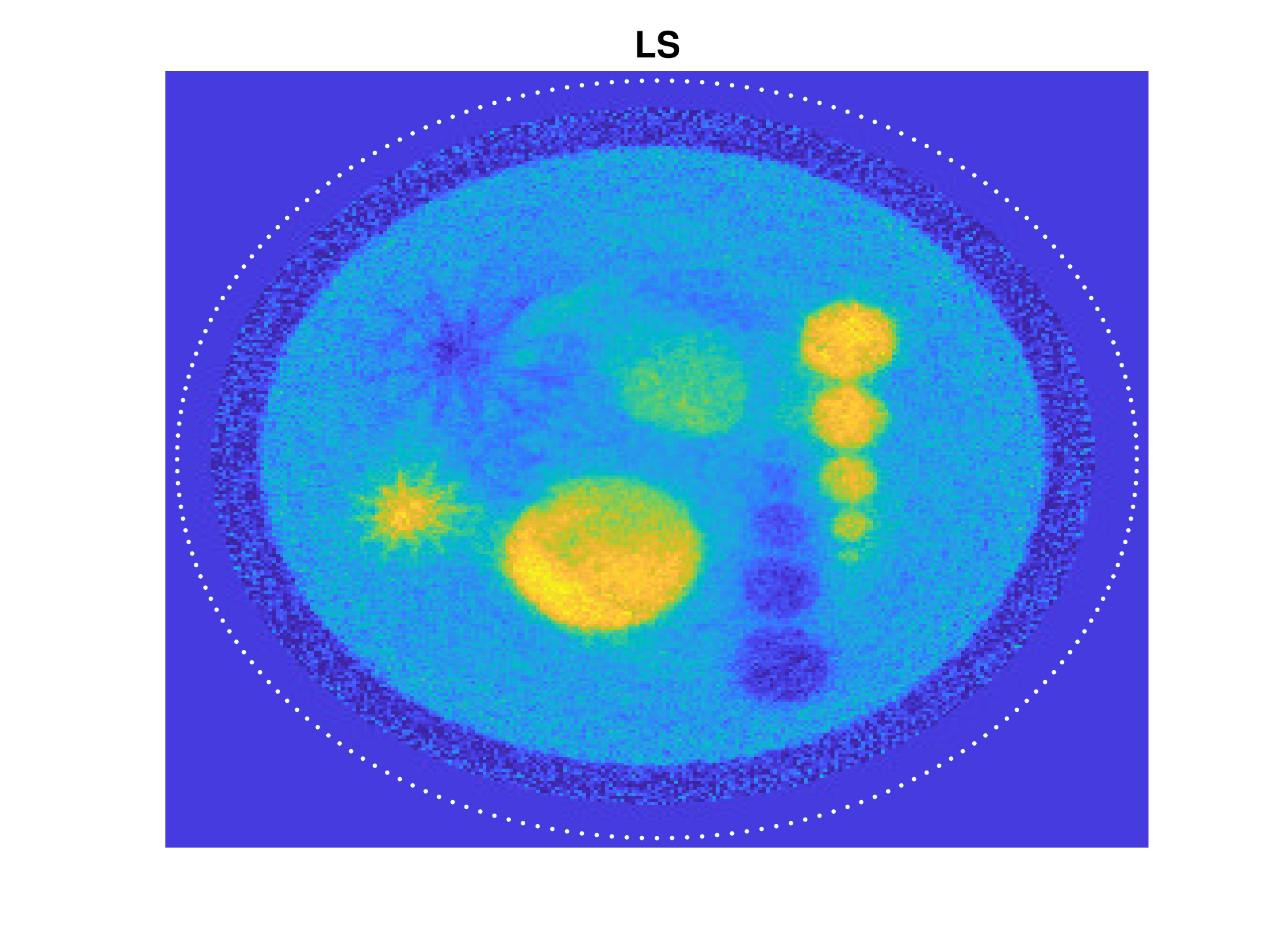}
\includegraphics[width=0.35\textwidth]{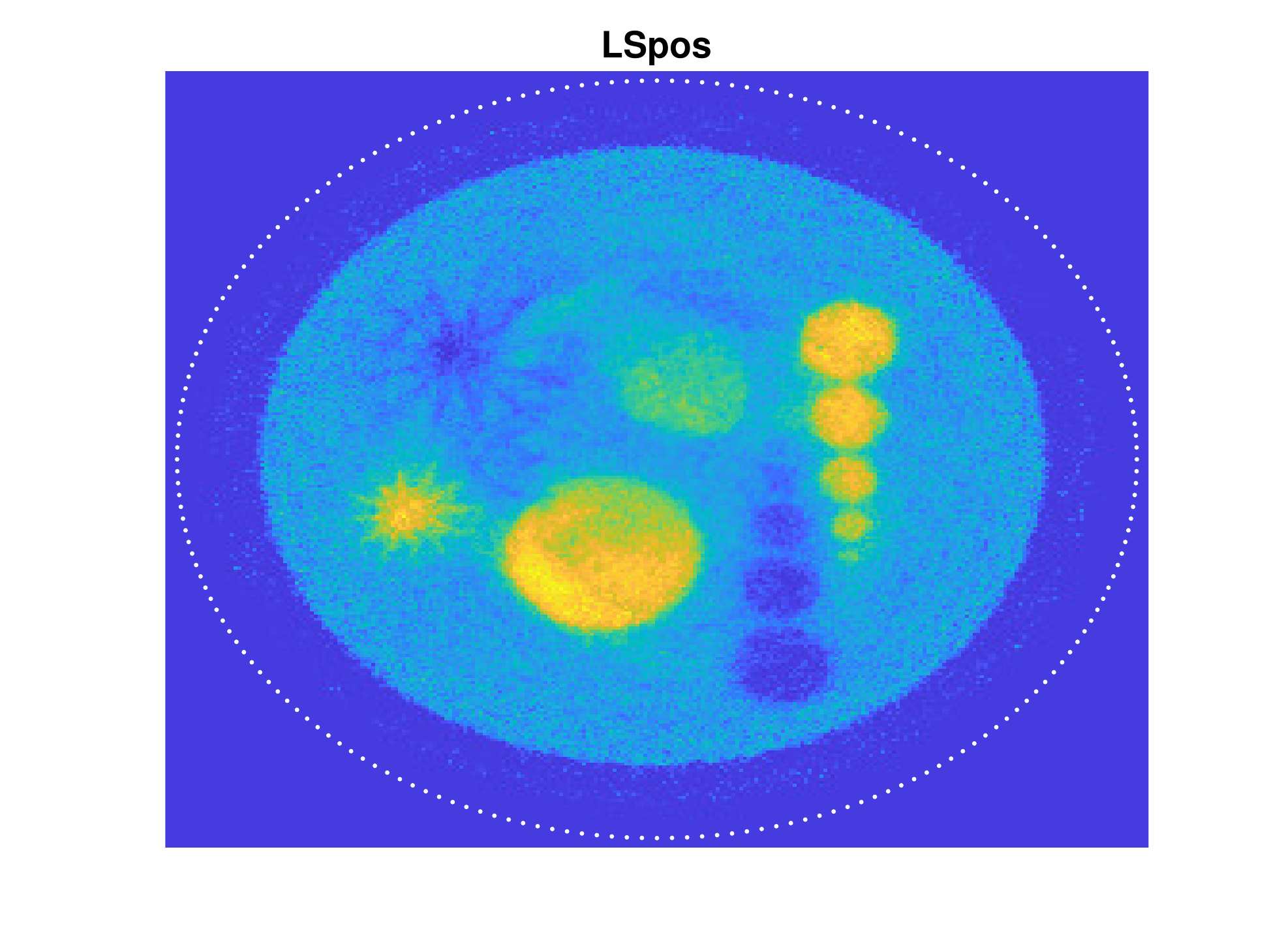}
\includegraphics[width=0.35\textwidth]{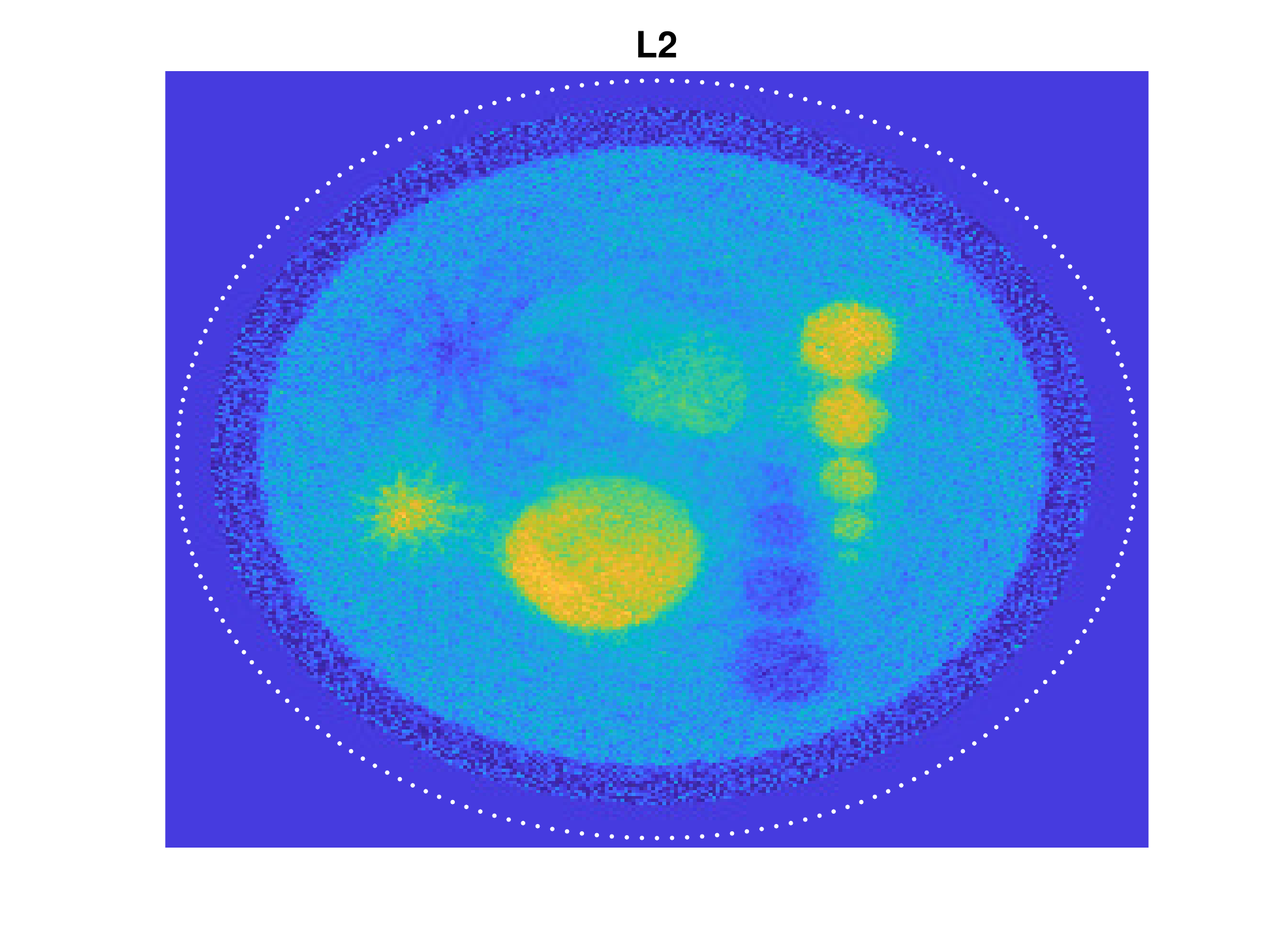}
\includegraphics[width=0.35\textwidth]{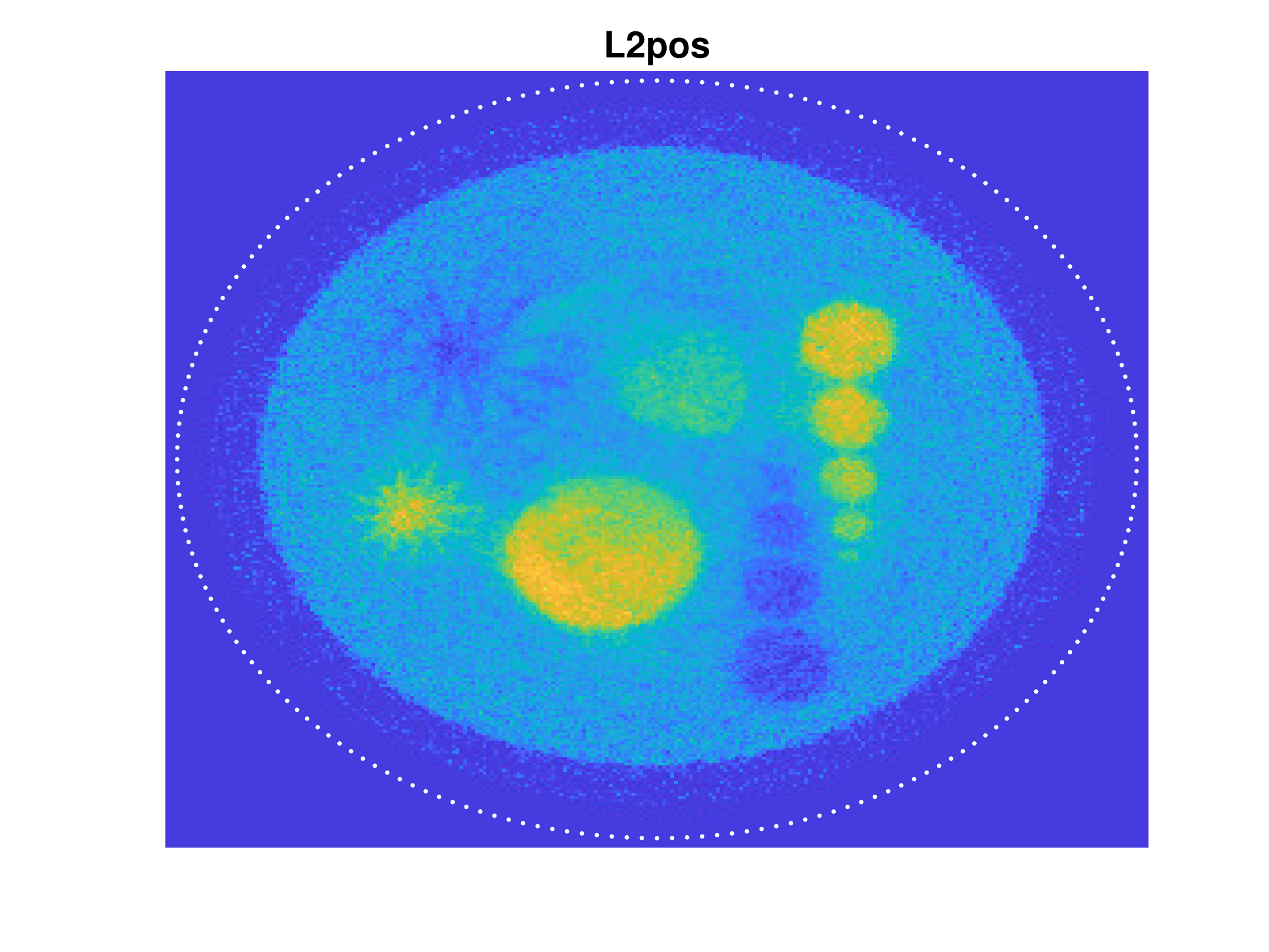}
\includegraphics[width=0.35\textwidth]{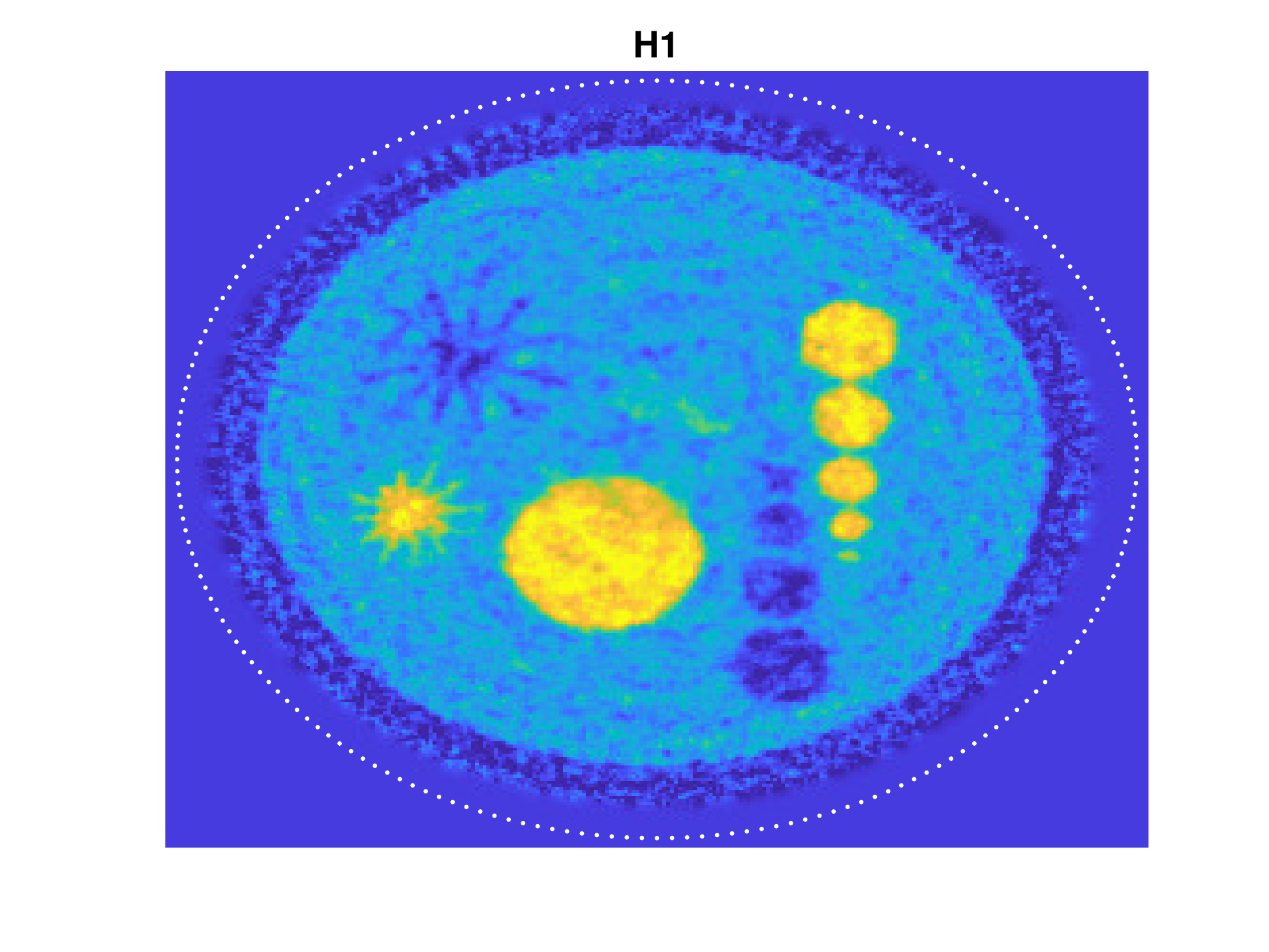}
\includegraphics[width=0.35\textwidth]{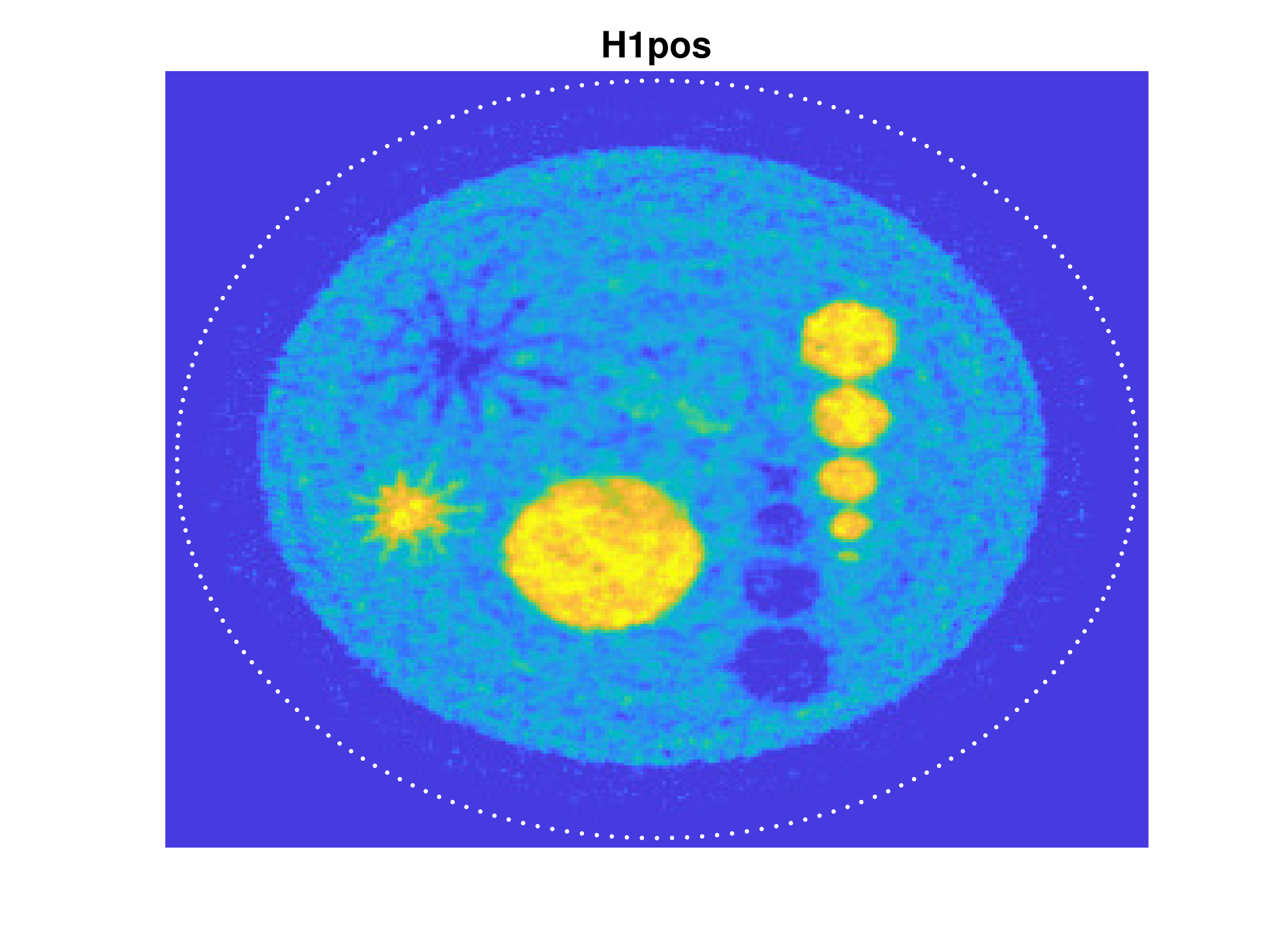}
\includegraphics[width=0.35\textwidth]{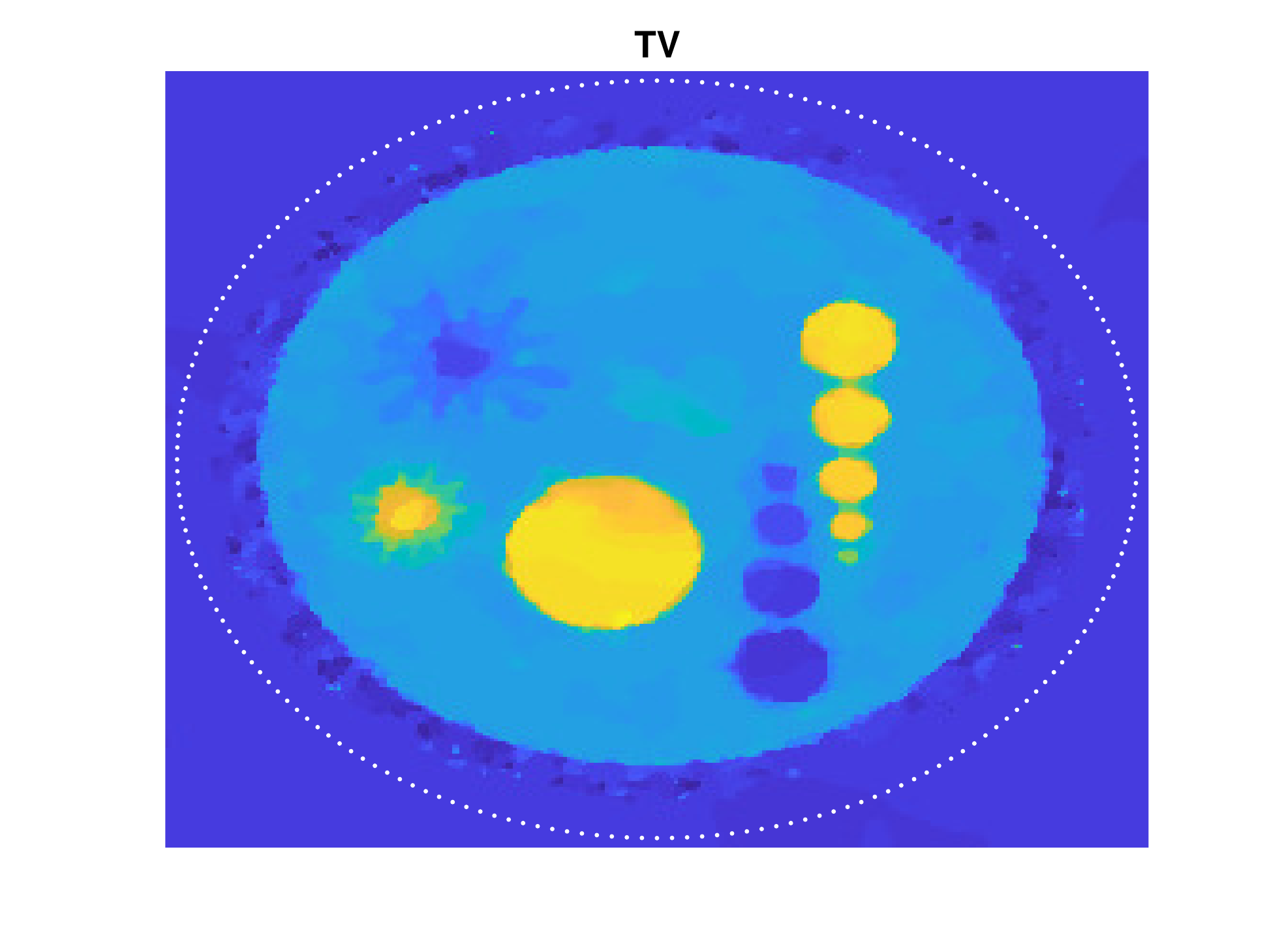}
\includegraphics[width=0.35\textwidth]{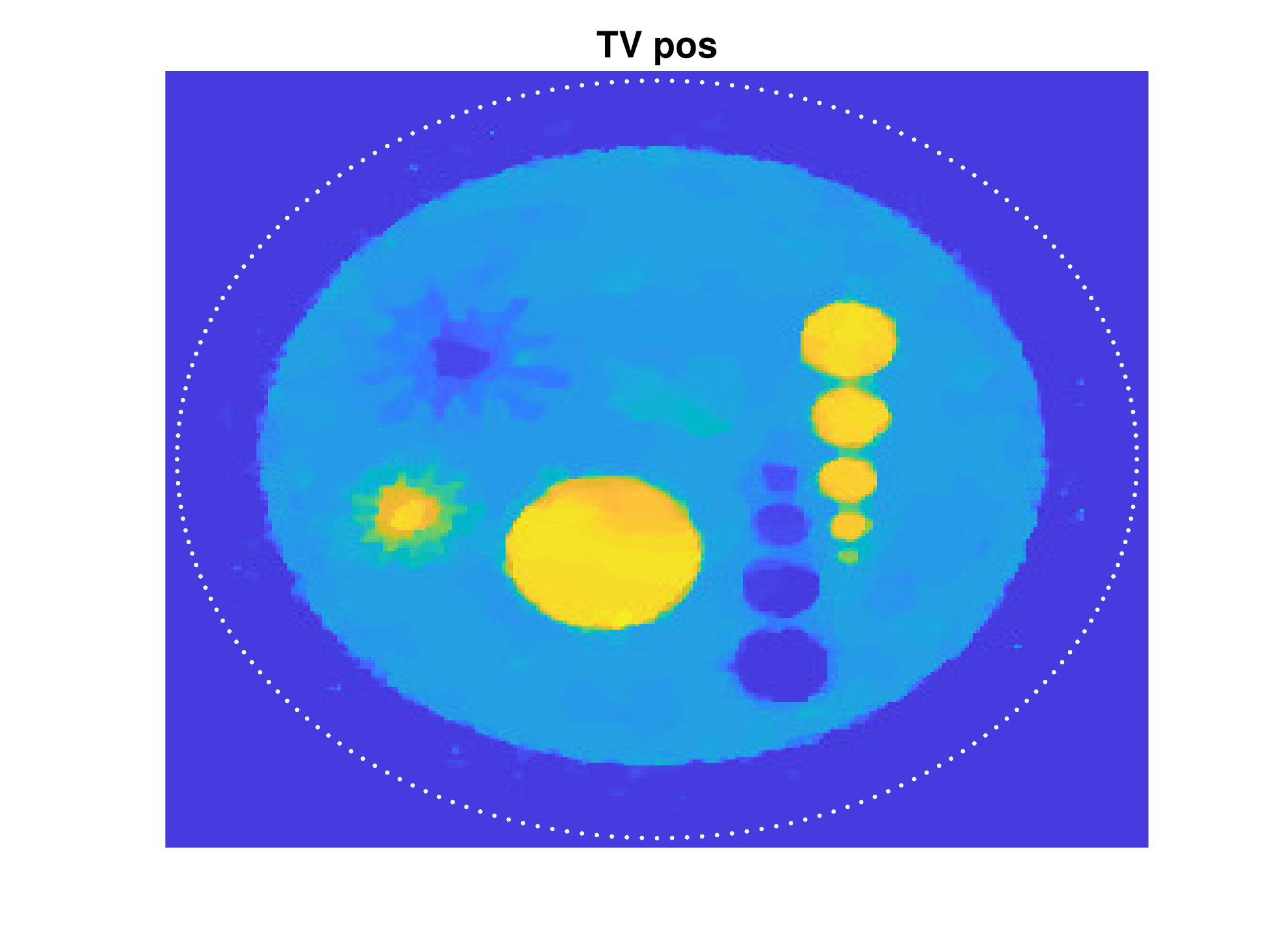}
\includegraphics[width=0.35\textwidth]{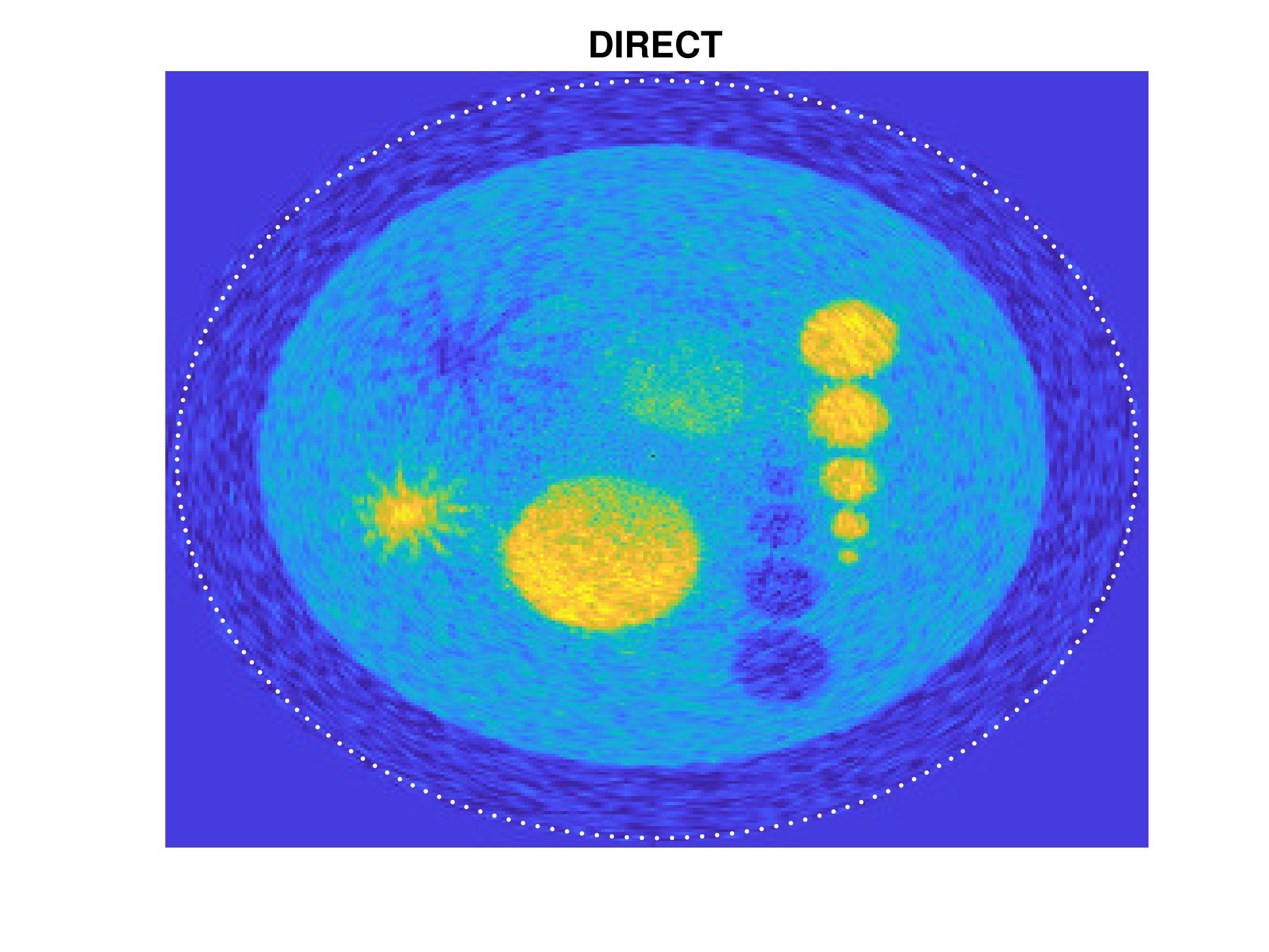}
\includegraphics[width=0.35\textwidth]{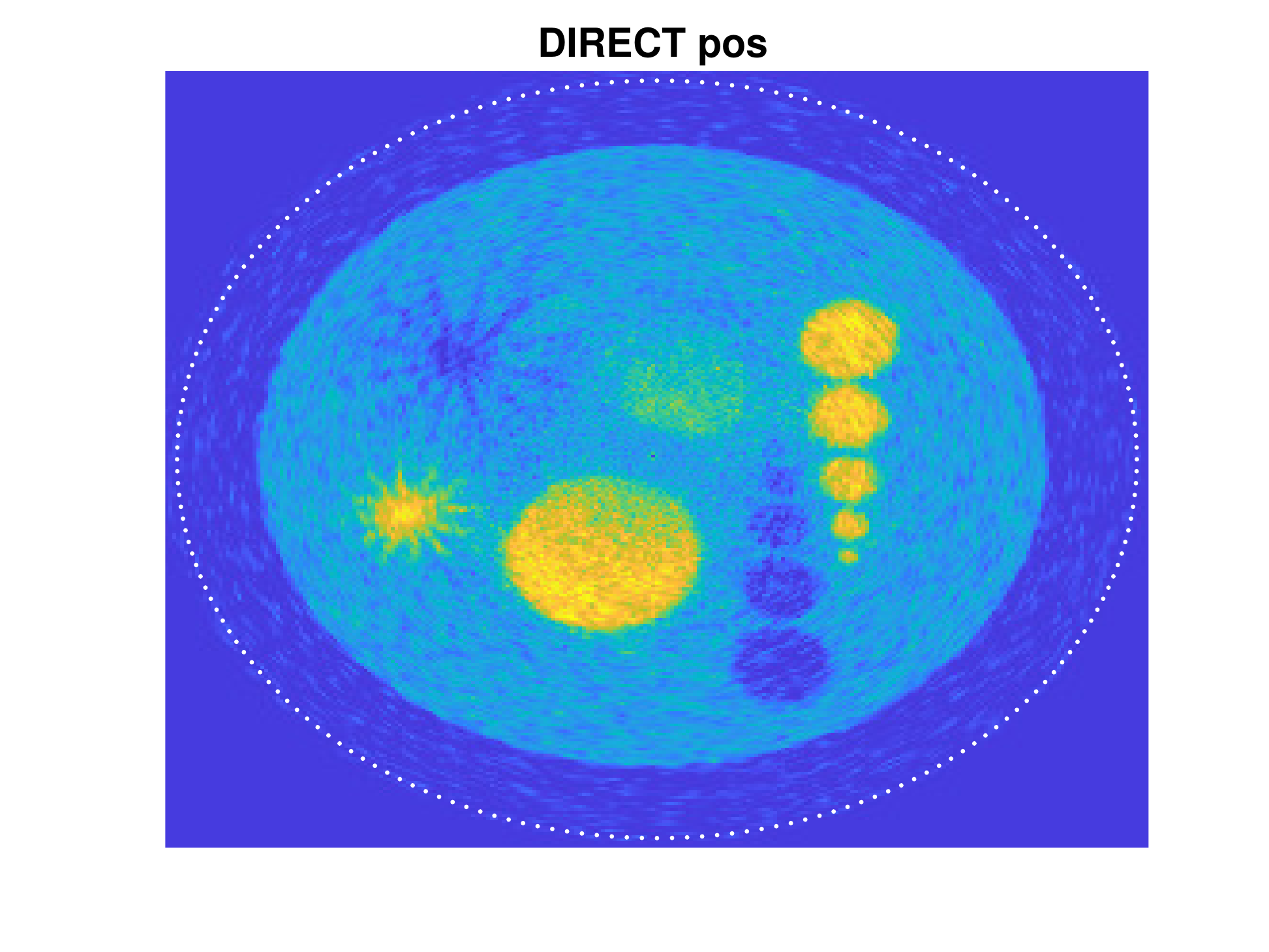}
\caption{\textbf{Reconstructions from noisy data.}
Top left: least squares.
Top right: positivity constraint least squares.
Line 2, left: $L^2$-regularization.
Line 2, right: $L^2$-regularization with positivity constraint.
Line 3, left: $H^1$-regularization.
Line 3, right: $H^1$-regularization with positivity constraint.
Line 4, left: TV-regularization.
Line 4, right: TV-regularization with positivity constraint.
Bottom,  left: Fourier reconstruction from  \cite{haltmeier2017inversion}.
Bottom,  right: Nonnegative part of Fourier reconstruction.}
\label{fig:recN}
\end{figure}

\begin{figure}\centering
\includegraphics[width=0.45\textwidth]{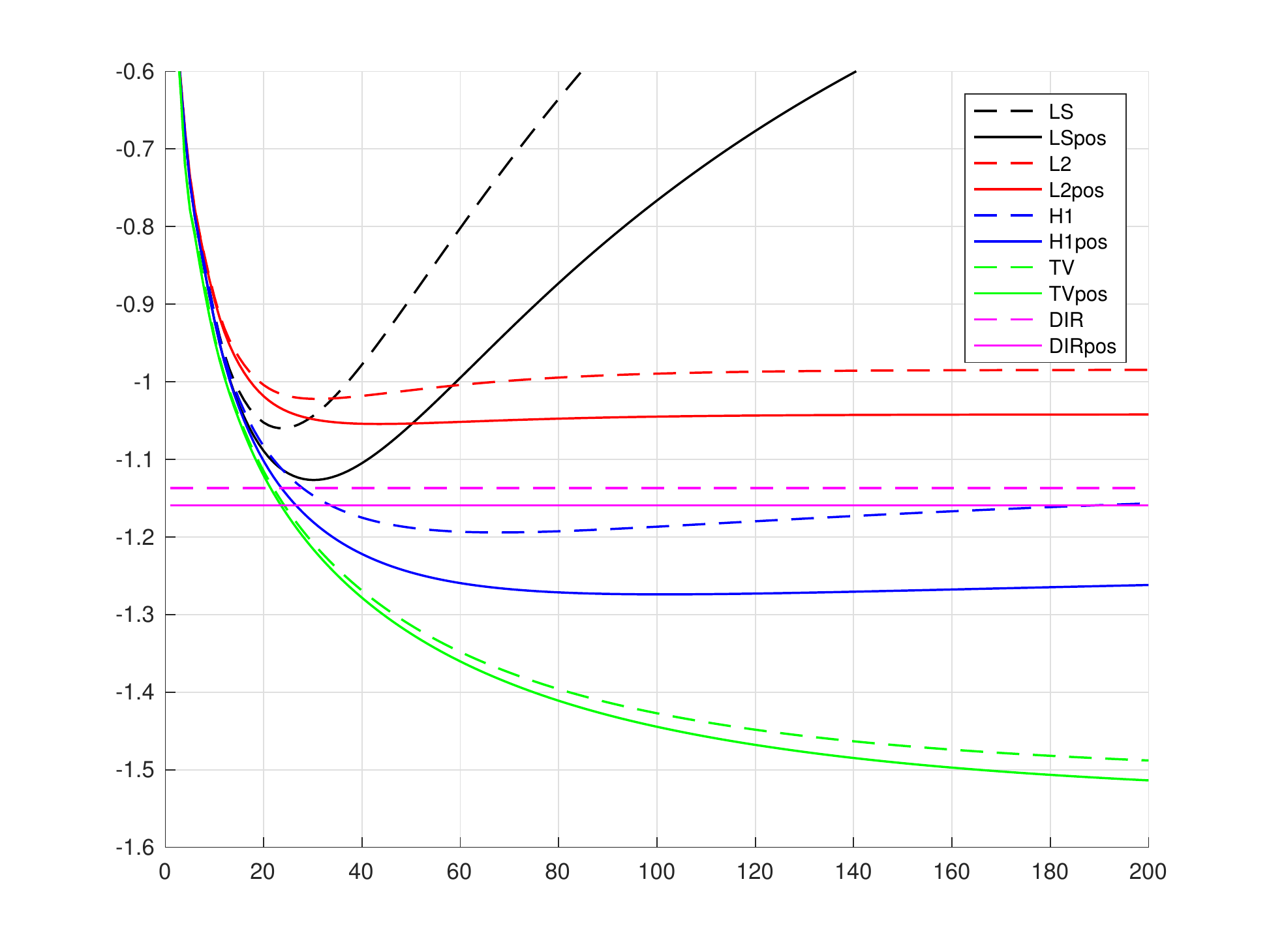}
\includegraphics[width=0.45\textwidth]{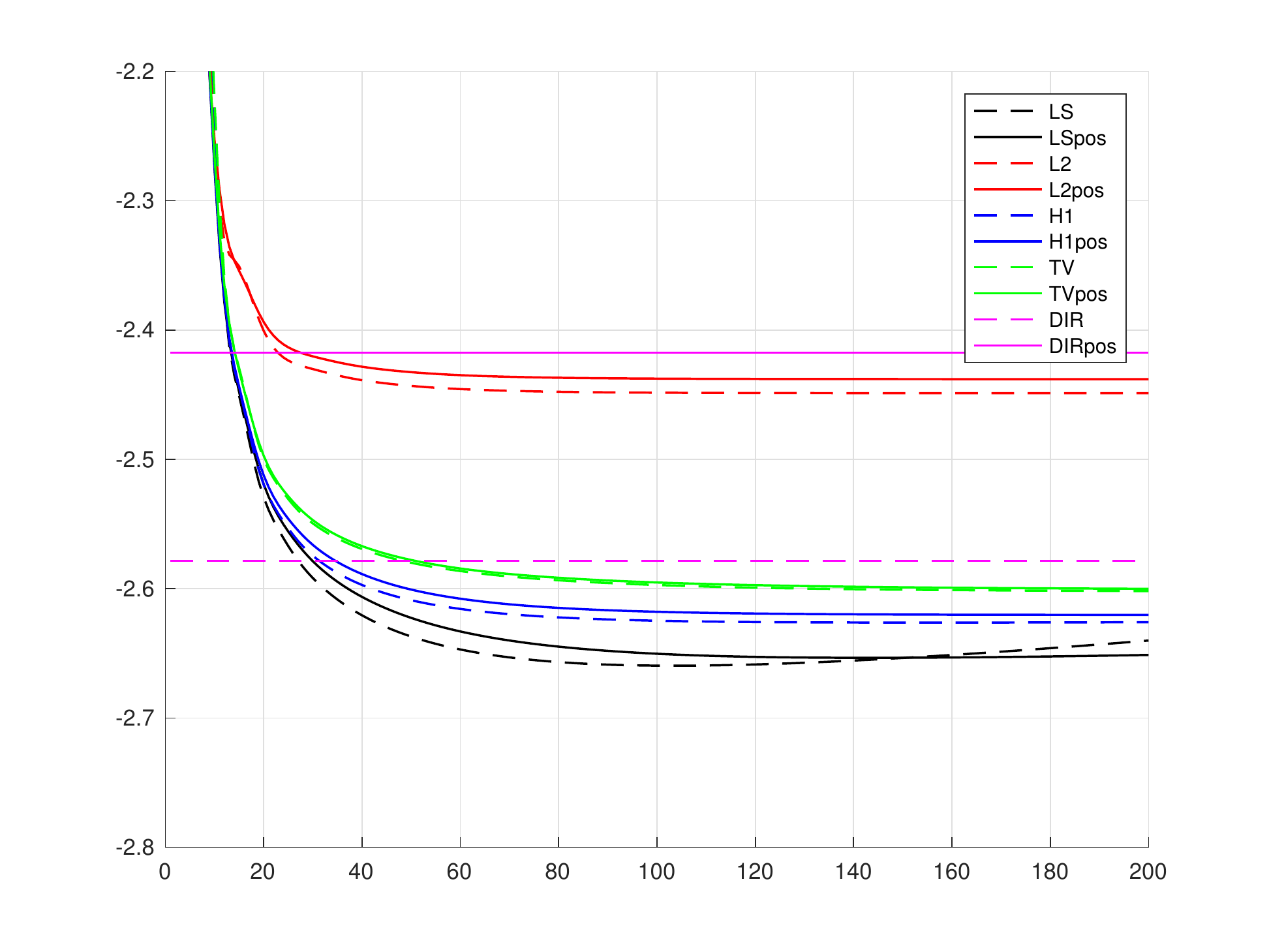}
\caption{\textbf{Reconstruction errors and residuals from noisy data.}
Left: Logarithm $\log_{10} E^2(\fnum_j)$ of squared  reconstruction error.
Right: Logarithm $\log_{10} R^2 (\fnum_j)$ of squared  residuals.}\label{fig:errN}
\end{figure}

\subsubsection*{Reconstruction from noisy  data}

To test the algorithms in a more realistic situation, we  repeated the
simulation studies where  we added additive Gaussian noise $\noisenum
\in \R^{(P+1)\times Q}$  to the simulated data.
The  relative $\ell^2$-error
\begin{equation*}
\delta = \frac{\sqrt{\sum_{k,\ell}\abs{\noisenum[k,\ell]}^2 }}{\sqrt{\sum_{k,\ell}\abs{\gnum[k,\ell]}^2}}
\end{equation*}
 is approximately $5 \%$.
 We again consider  $L^2$-regularization,  $H^1$-regularization, TV-regularization,
 least-squares  minimization and the Fourier reconstruction method from  \cite{haltmeier2017inversion}, where each method is used with and without positivity constraint.
 In order to stabilize the reconstruction, the regularization parameter  had to be increased
 for all methods. It has been taken  $\la = 0.14$ for $L^2$-regularization,
 $\la = 0.06$ for $H^1$-regularization, and  $\la = 0.015$ for TV-regularization.

 Reconstruction results from noisy data
 are shown in Figure~\ref{fig:recN}. We used 200 iterations
 for the regularized iterations  and  stopped the least squares iterations
 after 15 iterations in order to prevent overfitting. Logarithmic plots  of the
 squared relative $\ell^2$-error $E^2(\fnum_j) $ and the squared relative $\ell^2$-residual
 $R^2(\fnum_j) $ are shown in Figure \ref{fig:errN}. As can be seen from the reconstruction as well as the
 evolution of the reconstruction error, TV-regularization again performs
 best for the considered type of phantom. In particular, in all other methods,
 the  ghost image is clearly visible and the reconstruction error is much larger than
 for TV-regularization.

\section{Conclusion and outlook}
\label{sec:conclusion}

In this paper, we investigated
variational  regularization  methods for  the stable inversion of the conical Radon
transform with vertices on the sphere.
We presented  convergence results of variational  regularization
and a numerical minimization algorithm based on the
Chambolle-Pock primal dual algorithm  \cite{chambolle2011,sidky2012convex}.
In particular, the algorithm framework allows a TV-regularizer, quadratic penalties,
positivity  constraint as well as their combinations.
In terms of relative $\ell^2$-reconstruction error, we found TV-regularization
to clearly outperform least squares
quadratic-regularization (all with or without positivity constraint).

The results in this paper show that compared to quadratic regularization, the use of
non-smooth convex regularization can  clearly  improve reconstruction of the conical Radon transforms.
Several interesting generalizations of variational regularization are possible.  First, one could replace the sphere with more general
surfaces of vertices  and allow different axis for each  vertex location.
Second, the algorithmic framework can be extended to variable axis on the sphere.
The transform then depends on the $2n-1$ independent variables and is therefore
highly overdetermined.  In applications such as  emission tomography with Compton camera,
the noise level is usually high   and the full overdetermined transform must
actually be used for image reconstruction. Moreover, noise statistics as well as factors  governing system performance
such as Doppler broadening or polarization  \cite{rogers2004compton}  should be included
in the forward  and  adjoint  problem. Resulting algorithms based on
conical Radon transforms (which require  binning of Compton camera raw  data)
should be compared with  list mode EM reconstruction
algorithms  \cite{wilderman2001improved} that work without data binning.

Besides numerical investigations, several theoretical aspects have to be
addressed in future work. This includes stability and uniqueness analysis
as well as range descriptions of the various forms of conical
Radon transforms.

\section*{Acknowledgement}

The work of Markus Haltmeier  is supported through the Austrian Science Fund (FWF), project P 30747-N32.

\end{document}